\documentclass[12pt, a4paper]{article}

\usepackage[T1]{fontenc}
\usepackage{amsmath}
\usepackage{amsthm}
\usepackage{amssymb}
\usepackage{amsfonts}
\usepackage{graphicx}
\usepackage{authblk}
\usepackage{mathtools}
\usepackage{dsfont}
\usepackage{hyperref}
\usepackage{enumerate}

\theoremstyle{plain}
\newtheorem{thm}{Theorem}
\newtheorem{lem}[thm]{Lemma}
\newtheorem{pro}[thm]{Proposition}

\theoremstyle{definition}
\newtheorem{rem}[thm]{Remark}
\newtheorem{defi}[thm]{Definition}

\renewenvironment{proof}[1][]{\smallskip \noindent {\bf #1:}\hskip\labelsep}{\hspace*{\fill}$\square$\medskip\par} 
\newenvironment{mysmallmatrix}{\left(\begin{smallmatrix}}{\end{smallmatrix}\right)}

\def\R{\mathbb{R}}
\def\E{\mathbb{E}}
\def\phi{\varphi}
\def\eps{{\varepsilon}}
\def\diff{\mathrm{d}}

\hypersetup{
    pdftitle={Scaling limits for a random boxes model},
    pdfauthor={Aurzada and Schwinn},
    colorlinks=true,
    linkcolor=blue,
    citecolor=red,
}

\begin{document}

\date{\today}
\title{Scaling limits for a random boxes model}
\author[1]{Frank Aurzada}
\author[1,2]{Sebastian Schwinn}
\affil[1]{Department of Mathematics, Technische Universit\"at Darmstadt}
\affil[2]{Graduate School CE, Technische Universit\"at Darmstadt}
\maketitle

\begin{abstract}
We consider random rectangles in $\R^2$ that are distributed according to a Poisson random measure, i.e., independently and uniformly scattered in the plane.
The distributions of the length and the width of the rectangles are heavy-tailed with different parameters. We investigate the scaling behaviour of the related random fields as the intensity of the random measure grows to infinity while the mean edge lengths tend to zero.
We characterise the arising scaling regimes, identify the limiting random fields and give statistical properties of these limits.
\end{abstract}

\noindent {\bf 2010 Mathematics Subject Classification:} 60G60 (primary); 60F05, 60G55 (secondary).

\bigskip

\noindent {\bf Keywords:} Gaussian random field, generalised random field, Poisson point process, Poisson random field, random balls model, random grain model, random field, stable random field.

\section{Introduction}
\subsection{Model}
Let $B(x,u)$ denote the two-dimensional rectangular box in $\R^2$ with centre at $x$ and edge lengths $u_i$ for $i=1,2$.
We consider a family of rectangles $(B({X}^{(j)},U^{(j)}))_j$ in~$\R^2$ (also referred to as boxes) generated by a Poisson point process $(X^{(j)},U^{(j)})_j$ in $\R^2 \times \R_+^2$.
Let $N$ be a Poisson random measure with intensity measure given by $n(\diff x, \diff u) = \lambda \diff x F(\diff u)$, where the intensity~$\lambda$ is a positive constant.
The probability measure $F$ on $\R_+^2$ is given by 
\begin{equation}\label{eq:Fdu}
 F(\diff u)=c_Ff_1(u_1)f_2(u_2) \diff u_1 \diff u_2,
\end{equation}
where $c_F>0$ is the normalising constant and $f_i(u_i) \sim {1}/{u_i^{\gamma_i+1}}$ as $u_i \to \infty$ for $i=1,2$ with $\gamma_1 \in (1,2)$ and $\gamma_1 < \gamma_2$. 
Hence, we assume w.l.o.g.\ that the tail of the distribution of the length is heavier than that of the width.
Moreover, we assume for the sake of convenience that we have $c_F=1$ (because one could simply think that $c_F$ is included in $\lambda$ in the case of $c_F \neq 1$) and we write $f(y) \sim g(y)$ if ${f(y)}/{g(y)} \to 1$.
Note that for $i=1,2$
$$
\int_{\R_+} u_i f_i(u_i) \diff u_i < \infty,
$$
i.e., the expected length and the expected width (and thus area) of a box are finite.

We discuss random fields defined on certain spaces of signed measures.
Let us denote by $\mathcal{M}_2$ the linear space of signed measures $\mu$ on $\R^2$ with finite total variation $\Vert \mu \Vert \coloneqq \vert \mu \vert (\R^2) < \infty$, where $\vert \mu \vert$ is the total variation measure of~$\mu$ (see, e.g., \cite[p.\ 116]{rudin1987complex}).
We are interested in the cumulative volume induced by the boxes and measured by $\mu \in \mathcal{M}_2$.
Therefore, we define the random field $J \coloneqq (J(\mu))_{\mu}$ on $\mathcal{M}_2$ by
$$
J(\mu) \coloneqq \int_{\R^2 \times \R_+^2} \mu(B(x,u)) N(\diff x,\diff u).
$$
Since our purpose is to deal with centred random fields, we introduce the notation for the corresponding centred Poisson random measure $\widetilde{N}\coloneqq N-n$ and centred integral $\widetilde{J}(\mu) \coloneqq J(\mu)-\E J(\mu)$. 

The goal of this paper is to obtain scaling limits for the random field~$\widetilde{J}$.
By scaling, we mean that the length and the width of the boxes are shrinking to zero, i.e., the scaled edge lengths are $\rho u_i$ with scaling parameter $\rho \to 0$, and that the expected number of boxes is increasing, i.e., the intensity $\lambda$ of the Poisson point process is tending to infinity as a function of $\rho$.
The precise behaviour of $\lambda=\lambda(\rho) \to \infty$ is specified in the different scaling regimes below.
Following the notational convention from above, we denote by $\widetilde{J}_\rho$ the centred random field corresponding to the Poisson random measure $N_\rho$ with the modified intensity $\lambda_\rho \coloneqq \lambda(\rho)$ and scaled edge lengths, i.e., $F_\rho$ is the image measure of $F$ by the change $u \mapsto \rho u$.

Next, we want to say a few words about the applications of random balls models.
The motivation comes from models from telecommunication networks.
A list of some references can be found at the beginning of Chapter~3 in~\cite{MR3186156}.
In dimension $d=1$, the model applies to the random variation in packet network traffic, where the traffic is generated by independent sources over time. 
The quantity of interest is the limiting distribution of the aggregated traffic as the time and the number of sources both tend to infinity (possibly with different rates).
These different rates can result in different scaling regimes of the superposed network traffic.
In some papers, the `traffic' additionally has a weight, which can be interpreted as the amount of required resources, the transmission power or the file sizes (see, e.g., \cite{BRETON20093633, Fasen2010,Kaj2008,MR3186156}).
Our model can be interpreted in the same way, when the length of the rectangle is thought to be transmission time and the width a weight representing, e.g., a transmission rate.
Alternatively, our random rectangles model could model a simplified two-dimensional wireless network.
Imagine that there are spatially uniformly distributed stations which are equipped with emitters.
In our case, the range for transmission (with constant power) of each station is given by a rectangular area and the total power of emission is measured by $\mu \in \mathcal{M}_2$.

\subsection{Related work}
A basic reference on limit theorems of Poisson integrals is \emph{Random processes by example} by Lifshits \cite{MR3186156}. The main references for us are \cite{Bierme2010,kaj2007}.

Kaj et al.\ \cite{kaj2007} study the limits of a spatial random field generated by independently and uniformly scattered random sets in $\R^d$. The sets (also referred to as grains) have a random volume but a predetermined shape.
The size of a grain is given by a \emph{single} heavy-tailed distribution, i.e., scaling means that the intensity~$\lambda$ grows to infinity while the mean volume~$\rho$ of the sets tends to zero.
They obtain three different limits depending on the relative speed at which $\lambda$ and $\rho$ are scaled.
Furthermore, they provide statistical properties of the limits.

In \cite{Bierme2010}, Bierm\'{e} et al.\ consider a random balls model of germ-grain type as well.
The predetermined shape of the grains is a ball, whose size depends on the scaling parameter $\rho$ and the random radius.
The radius distribution has a power-law behaviour either in zero or at infinity, i.e., they deal with zooming in and zooming out.
As main result, they can construct all self-similar, translation and rotation invariant Gaussian fields through zooming procedures in the random balls model. 

Breton and Dombry \cite{BRETON20093633} investigate weighted random balls models.
There, the balls additionally have random weights, whose law belongs to the normal domain of attraction of the $\alpha$-stable distribution with $\alpha \in (1,2]$.
They obtain different limiting random fields depending on the regimes and give statistical properties.

An anisotropic scaling is examined by Pilipauskait\.e and Surgailis \cite{pilipauskaite2016}.
They study the scaling limits of the random grain model on the plane with heavy-tailed grain area distribution.
The anisotropy is implemented by scaling the $x$- and $y$-direction at different rates.
Therefore, in the case of the grains being rectangles, the ratio of the edge lengths of \emph{all} rectangles tends either to  zero or to infinity under the scaling.
This property distinguishes their model from our random boxes model, where each rectangle has a random length-to-width ratio that does not change under the scaling.

Moreover, there are much more related papers that investigate limits of random balls models (e.g., \cite{GOBARD20151284}) and in particular of `teletraffic' models (see~\cite{Fasen2010,Fay2006,Kaj2008} and a list of further references in~\cite{MR3186156}).

\subsection{Overview}
In a nutshell, our paper extends the work from Bierm\'{e} et al.\ \cite{Bierme2010} and Kaj et al.~\cite{kaj2007} to a random boxes model where the size of a grain depends on \emph{two} differently heavy-tailed distributed random variables instead of just one random variable for the volume of the grain.
To be more precise, the shape of the grains is rectangular with a random length and a random width (mutually independent).
Therefore, our model differs from those in that the volume is given by the product of the length and the width, and each box simultaneously gets a random length-to-width ratio.
As a consequence, the main novelty of this work is that our random boxes model leads to a greater number of scaling regimes than other random balls models (e.g., \cite{Bierme2010,BRETON20093633,kaj2007}).
In particular, the so-called Poisson-lines scaling regime with its distinctive graphical representation has not arisen so far (see Section~\ref{sec:poisline}).
The class of limiting random fields contains linear random fields that are Gaussian, compensated Poisson integrals and integrals with respect to a stable random measure.

Let us outline different scaling regimes which result in different limits.
As mentioned above, the scaling regimes are defined by the joint behaviour of the scaling parameter $\rho$ and the intensity $\lambda_\rho$ of the Poisson point process as $\rho \to 0$.
We distinguish the following regimes:
\begin{itemize}
 \item High intensity regime: $\lambda_\rho \rho^{\gamma_1+\gamma_2} \to \infty$.
 \item Intermediate intensity regime: $ \lambda_\rho \rho^{\gamma_1+\gamma_2} \to a \in (0,\infty)$.
 \item Low intensity regime:  $\lambda_\rho \rho^{\gamma_1+\gamma_2} \to 0$.
\end{itemize}
The low intensity regime has to be divided once more into three different sub-regimes.
Our naming of these sub-regimes is based on the limits and on the objects that can be spotted in a graphical representation.
We distinguish the following sub-regimes:
\begin{itemize}
 \item Gaussian-lines scaling regime: $\lambda_\rho \rho^{\gamma_1+\eta} \to a \in (0,\infty)$ for some constant~\mbox{$\eta \in (0,\gamma_2)$} and thus $ \lambda_\rho \rho^{\gamma_1} \to \infty$. With regard to the scaling limit, it is of no importance to take care of the precise behaviour of~$\lambda_\rho \rho^{\gamma_2}$ (as long as $\lambda_\rho \rho^{\gamma_1+\gamma_2} \to 0$).
 \item Poisson-lines scaling regime: $ \lambda_\rho \rho^{\gamma_1} \to a \in (0,\infty)$ and thus $\lambda_\rho \rho^{\gamma_2} \to 0$.
 \item Points scaling regime: $\lambda_\rho \rho^{\gamma_1} \to 0$.
\end{itemize}
So far, we have assumed $\gamma_1 < 2$.
For $2 < \gamma_1 \leq \gamma_2$, the length and the width of the boxes have finite variances.
In this case, there is only one scaling limit and we just require that $\lambda_\rho \to \infty$ as $\rho \to 0$, i.e., there is no further condition on the joint behaviour of $\rho$ and $\lambda_\rho$.

The remainder of this paper is structured as follows:
Section~\ref{sec:main} contains the theorems of convergence to the limiting random fields (subdivided into the different scaling regimes in Sections~\ref{sec:main_high}--\ref{sec:main_low}, respectively), a comparison to the model where the length and the width of the boxes have finite variances (Section~\ref{sec:finitevar}), and further facts on statistical properties of the limits as well as a modified model with randomly rotated boxes (Section~\ref{sec:statprop}).
We collect some preliminaries in Section~\ref{sec:preliminaries} in order to prove the main results in Section~\ref{sec:proofmain}.

\section{Main results}\label{sec:main}
The following results are theorems of convergence of the finite-dimensional distributions of the centred and renormalised random field
$$
\left( \frac{\widetilde{J}_\rho(\mu)}{n_\rho} \right)_{\mu \in \mathcal{M}}
$$
to a limiting random field, where the corresponding space of signed measures~$\mathcal{M}$ and the function $n_\rho \coloneqq n(\rho)$ are defined in the theorems below, respectively.
We denote this convergence by $\frac{\widetilde{J}_\rho(\cdot)}{n_\rho} \xrightarrow{\mathcal{M}} W(\cdot)$, where in each case the limiting random field $(W(\mu))_\mu$ is specified there.

\subsection{High intensity regime}\label{sec:main_high}
We look at the high intensity regime where $\lambda_\rho \rho^{\gamma_1+\gamma_2} \to \infty$.
First, we define the space of signed measures $\mathcal{M}^{\gamma_1,\gamma_2}$ where the theorem of convergence holds.
\begin{defi}\label{def:M_high}
Let $\mathcal{M}^{\gamma_1,\gamma_2}$ be the subset of $\mathcal{M}_2$ with the following property:
For each $\mu \in \mathcal{M}^{\gamma_1,\gamma_2}$, there exist constants $C>0$ and $\alpha_i$ with $\gamma_i < \alpha_i \leq 2$ for $i=1,2$ such that for all $u \in \R_+^2$
\begin{equation}\label{eq:phibound}
 \int_{\R^2}  \mu (B(x,u)) ^2 \diff x \leq C \min\left(u_1,u_1^{\alpha_1} \right) \min\left(u_2,u_2^{\alpha_2} \right).
\end{equation}
\end{defi}

The limiting random field is given by a centred Gaussian linear random field. 
\begin{thm}\label{thm:high}
Let $\gamma_i \in (1,2)$ for $i=1,2$, $\lambda_\rho \to \infty$ and $\lambda_\rho \rho^{\gamma_1+\gamma_2} \to \infty$ as $\rho \to 0$.
Then, we have
$$
\frac{\widetilde{J}_\rho(\cdot)}{\sqrt{ \lambda_\rho \rho^{\gamma_1+\gamma_2}}} \xrightarrow{\mathcal{M}^{\gamma_1,\gamma_2}}  Z(\cdot)
$$
as $\rho \to 0$, where $(Z(\mu))_\mu$ is the centred Gaussian linear random field with covariance function
\begin{equation}\label{eq:cov_high_gaussian}
 C_Z(\mu,\nu)= \int_{\R^2 \times \R_+^2} \mu(B(x,u))\nu(B(x,u)) \frac{1}{u_1^{\gamma_1+1}} \frac{1}{u_2^{\gamma_2+1}} \diff(x,u). 
\end{equation}
\end{thm}

\subsection{Intermediate intensity regime}\label{sec:main_inter}
In the intermediate intensity regime where $ \lambda_\rho \rho^{\gamma_1+\gamma_2} \to a \in (0,\infty)$, the space of signed measures is identical with the one in the high intensity regime.
The limiting random field consists of compensated Poisson integrals.

\begin{thm}\label{thm:inter}
Let $\gamma_i \in (1,2)$ for $i=1,2$, $\lambda_\rho \to \infty$ and $\lambda_\rho \rho^{\gamma_1+\gamma_2} \to 1$ as $\rho \to 0$.
Then, we have
$$
\widetilde{J}_\rho(\cdot) \xrightarrow{\mathcal{M}^{\gamma_1,\gamma_2}} J_I(\cdot)
$$
as $\rho \to 0$, where $ (J_I(\mu))_\mu$ is the linear random field of compensated Poisson integrals
$$
J_I(\mu) \coloneqq \int_{\R^2 \times \R_+^2} \mu(B(x,u)) \widetilde{N}_I(\diff x,\diff u),
$$
where the intensity measure is given by $n_I(\diff x,\diff u)= \diff x \frac{1}{u_1^{\gamma_1+1}}  \frac{1}{u_2^{\gamma_2+1}} \diff u_1 \diff u_2$.
\end{thm}

We refer to Remark \ref{rem:inter_a} below for the result in the (general) intermediate intensity regime with $\lambda_\rho \rho^{\gamma_1+\gamma_2} \to a \in (0,\infty)$ as $\rho \to 0$, where $a$ not necessarily equals $1$.

\subsection{Low intensity regime}\label{sec:main_low}
The low intensity regime is defined by $\lambda_\rho \rho^{\gamma_1+\gamma_2} \to 0$, which is divided once more into three different sub-regimes. In these sub-regimes, we additionally have to assume that the density function of the length of a box for \emph{small} values is bounded, i.e., we assume that there is some $c_{f_1}>0$ such that the inequality
\begin{equation}\label{eq:addrequirelow}
 f_1(u_1) \leq \frac{c_{f_1}}{u_1^{\gamma_1+1}}
\end{equation}
holds for \emph{all} $u_1 \in \R_+$.
This technical assumption ensures the existence of a suitable majorant for $f_1$ in the proofs below.
From now on, we treat the three sub-regimes separately.

\subsubsection{Gaussian-lines scaling regime}
We define the space of signed measures $\mathcal{M}_{L}$ for the Gaussian-lines scaling regime where $\lambda_\rho \rho^{\gamma_1+\eta} \to a \in (0,\infty)$ for some $\eta \in (0,\gamma_2)$.
\begin{defi}\label{def:M_lowgaussian}
Let $\mathcal{M}_{L}$ be the subset of $\mathcal{M}_2$ where 
\begin{itemize}
 \item each $\mu \in \mathcal{M}_{L}$ has a density function $f_\mu$, i.e., $\mu(\diff x) = f_\mu(x) \diff x$;
 \item for each $\mu \in \mathcal{M}_{L}$ the density function $f_\mu$ is bounded and decays at least exponentially fast, i.e., there exist constants $C_\mu>0$ and $c_\mu>0$ such that for all $x \in \R^2$
 \begin{equation}\label{eq:decay_mu}
  \vert f_\mu(x) \vert \leq C_\mu e^{-c_\mu(\vert x_1 \vert + \vert x_2 \vert)};
 \end{equation}
 \item for each $\mu \in \mathcal{M}_{L}$ the pointwise convergence
 \begin{equation}\label{eq:low_pointwiseconv}
  \frac{1}{\eps} \int_{B\left( x, \begin{mysmallmatrix} u_1 \\ \eps \end{mysmallmatrix}\right)} f_\mu(y) \diff y \to \int_{\left[x_1-\frac{u_1}{2},x_1+\frac{u_1}{2}\right]} f_\mu(y_1,x_2) \diff y_1  
 \end{equation}
 as $\eps \to 0$ holds for all $(x,u_1) \in \R^2 \times \R_+$.
\end{itemize}
\end{defi}

In the Gaussian-lines scaling regime, we require a further condition on the `lighter' tail index $\gamma_2$, namely $\gamma_2>2$.
Consequently, the width of a box has a finite variance.
The limiting random field is a centred Gaussian linear random field.

\begin{thm}\label{thm:lowgaus}
Let $\gamma_1 \in (1,2)$, $\gamma_2>2$, $\lambda_\rho \to \infty$ and $\lambda_\rho \rho^{\gamma_1+\eta} \to 1$ for some $\eta \in (0,\gamma_2)$ as $\rho \to 0$.
Then, we have
$$
\frac{\widetilde{J}_\rho(\cdot) }{{\rho}^{1-\eta / 2}} \xrightarrow{\mathcal{M}_{L}} Y(\cdot)
$$
as $\rho \to 0$, where $(Y(\mu))_\mu$ is the centred Gaussian linear random field with covariance function
\begin{equation}\label{eq:cov_low_gaussian}
 C_Y(\mu,\nu) = \int\limits_{\R^2 \times \R_+^2} \int\limits_{\left[x_1-\frac{u_1}{2},x_1+\frac{u_1}{2}\right]^2} f_\mu(y_1,x_2) f_\nu(y_2,x_2) \diff y \frac{u_2^2 f_2(u_2)}{u_1^{\gamma_1+1}} \diff (x,u).
\end{equation}
\end{thm}

We refer to Remark \ref{rem:lowgaus_a} below for the result in the (general) Gaussian-lines scaling regime with $\lambda_\rho \rho^{\gamma_1+\eta} \to a \in (0,\infty)$ as $\rho \to 0$, where $a$ not necessarily equals $1$.

\subsubsection{Poisson-lines scaling regime}\label{sec:poisline}
In the Poisson-lines scaling regime where $ \lambda_\rho \rho^{\gamma_1} \to a \in (0,\infty)$, we provide the theorem of convergence to a random field consisting of compensated Poisson integrals.
The corresponding space of signed measures coincides with the one from the Gaussian-lines scaling regime.

\begin{thm}\label{thm:lowpois}
Let $\gamma_1 \in (1,2)$, $\gamma_1 < \gamma_2$, $\lambda_\rho \to \infty$ and $\lambda_\rho \rho^{\gamma_1} \to 1$ as $\rho \to 0$.
Then, we have
$$
\frac{\widetilde{J}_\rho(\cdot) }{\rho} \xrightarrow{\mathcal{M}_{L}} J_L(\cdot)
$$
as $\rho \to 0$, where $(J_L(\mu))_\mu$ is the linear random field of compensated Poisson integrals
\begin{equation}\label{eq:poissonlineslimit}
 J_L(\mu) \coloneqq \int_{\R^2 \times \R_+^2} \left( u_2 \int_{\left[x_1-\frac{u_1}{2},x_1+\frac{u_1}{2}\right]} f_\mu(y_1,x_2) \diff y_1 \right) \widetilde{N}_L(\diff x,\diff u),
\end{equation}
where the intensity measure is given by
\begin{equation}\label{eq:int_meas_lines}
 n_L(\diff x,\diff u)= \diff x \frac{1}{u_1^{\gamma_1+1}} \diff u_1   f_2(u_2) \diff u_2.
\end{equation}
\end{thm}

We refer to Remark \ref{rem:lowpois_a} below for the result in the (general) Poisson-lines scaling regime with $\lambda_\rho \rho^{\gamma_1} \to a \in (0,\infty)$ as $\rho \to 0$, where $a$ not necessarily equals $1$.

\begin{figure}[ht]
\begin{center}
\includegraphics[trim=5cm 5cm 4.5cm 5cm, clip=true, width=0.4\linewidth]{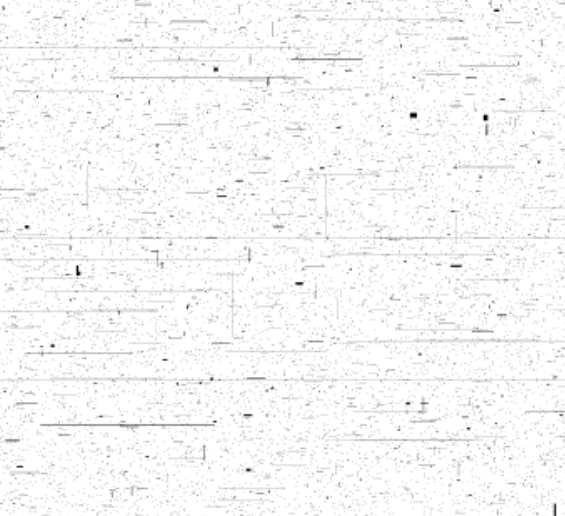}
\hspace{0.1\linewidth}
\includegraphics[trim=5cm 5cm 4.5cm 5cm, clip=true, width=0.4\linewidth]{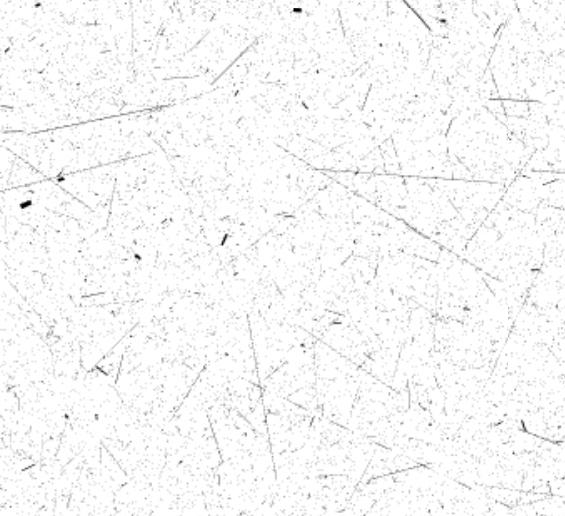}
\caption{Poisson-lines scaling regime.}\label{fig:lowpoisson}
\end{center}
\end{figure}

In the Poisson-lines scaling regime, we have $ \lambda_\rho \rho^{\gamma_1} \to a \in (0,\infty)$ and  $ \lambda_\rho \rho^{\gamma_2} \to 0$ as $\rho \to 0$.
This indicates a different behaviour for the length and the width of the boxes.
For a graphical representation, we ran simulations of the Poisson point processes for some small $\rho$ and appropriate $\lambda_\rho$.
We generated random Poisson points, where we chose Pareto distributions for the length and the width of the boxes.
Then, we plotted the boxes that are filled with black colour.
Two samples of the random boxes model in the Poisson-lines scaling regime are given in Figure~\ref{fig:lowpoisson}.
Besides points, we spot horizontal lines in the sample on the left hand side.
In the sample on the right hand side, each box is just additionally randomly rotated (cf.\ Section~\ref{sec:statprop} below for the definition of this modified model). 

\subsubsection{Points scaling regime}
In the points scaling regime where \mbox{$\lambda_\rho \rho^{\gamma_1} \to 0$}, we investigate the scaling behaviour of $\widetilde{J}_\rho$ on the space of signed measures $\mathcal{M}_{P}$ which is given as follows:

\begin{defi}\label{def:M_lowpoint}
Let $\mathcal{M}_{P}$ be the subset of $\mathcal{M}_2$ where
\begin{itemize}
 \item each signed measure $\mu \in \mathcal{M}_{P}$ has a \emph{continuous} density function $f_\mu$, i.e., $\mu(\diff x) = f_\mu(x) \diff x$;
 \item for each $\mu \in \mathcal{M}_{P}$ the density function $f_\mu$ is bounded and decays at least exponentially fast, i.e., there exist constants $C_\mu>0$ and $c_\mu>0$ such that for all $x \in \R^2$
 $$
 \vert f_\mu(x) \vert \leq C_\mu e^{-c_\mu(\vert x_1 \vert + \vert x_2 \vert)}.
 $$
\end{itemize}
\end{defi}

The limiting random field consists of integrals with respect to an $\alpha$-stable random measure.
For $\alpha \in (1,2)$, we denote by $\Lambda_\alpha$ the independently scattered $\alpha$-stable random measure with unit skewness and Lebesgue control measure (cf., e.g., \cite{samorodnitsky2000stable}).
We define the random linear functional
\begin{equation}\label{eq:def_stable_int}
 S_{\gamma_1}(\mu) \coloneqq \int_{\R^2} f_\mu(x) \Lambda_{\gamma_1}(\diff x), \quad \mu \in \mathcal{M}_{P},
\end{equation}
by its characteristic function at $1$
$$
 \E \left(e^{i S_{\gamma_1}(\mu) }\right)=  \exp \left( -\sigma_\mu^{\gamma_1} \left(1-i \beta_\mu \tan \left( \frac{\pi \gamma_1}{2} \right) \right) \right),
$$
where
\begin{equation}\label{eq:sigma_mu}
 \sigma_\mu=\Vert f_\mu \Vert_{\gamma_1}, \quad \beta_\mu= \Vert f_\mu \Vert_{\gamma_1}^{-\gamma_1} \left(\Vert {f_\mu}_+ \Vert_{\gamma_1}^{\gamma_1} - \Vert {f_\mu}_- \Vert_{\gamma_1}^{\gamma_1} \right)
\end{equation}
and ${f_\mu}_+ \coloneqq \max\left({f_\mu},0 \right)$, ${f_\mu}_- \coloneqq - \min\left({f_\mu},0\right)$.

\begin{thm}\label{thm:lowpoint}
Let $\gamma_1 \in (1,2)$, $\gamma_1 < \gamma_2$, $\lambda_\rho \to \infty$ and $\lambda_\rho \rho^{\gamma_1} \to 0$ as $\rho \to 0$.
Then, we have
$$
\frac{\widetilde{J}_\rho(\cdot) }{c_{\gamma_1,\gamma_2} \lambda_\rho^{{1}/{\gamma_1}} \rho^2} \xrightarrow{\mathcal{M}_{P}} S_{\gamma_1}(\cdot)
$$
as $\rho \to 0$, where the linear random field of functionals $(S_{\gamma_1}(\mu))_\mu$ and the constant $c_{\gamma_1,\gamma_2}$ are defined in~\eqref{eq:def_stable_int} and~\eqref{eq:c1c2_gamma_stable} below, respectively.
\end{thm}

We emphasise that the `heavier' tail index $\gamma_1$ for the length of a box appears primarily in the limit, i.e., the `lighter' tail index $\gamma_2$ only enters into a constant.
More precisely, the limit $S_{\gamma_1}(\mu)$ is a $\gamma_1$-stable random variable and the constant $c_{\gamma_1,\gamma_2}$ given in~\eqref{eq:c1c2_gamma_stable} below is the only quantity depending on the tail index~$\gamma_2$.
This contrasts the limits in the high and intermediate intensity regimes, where both parameters $\gamma_1$ and $\gamma_2$ are present in a homogeneous way in each limit. 

\subsection{The finite variance case}\label{sec:finitevar}
Finally, we want to investigate the scaling behaviour in the case where the area of a box has a finite variance. We assume that the length and the width of the boxes have finite second moments instead of heavy tails.
Similar to above, let $F$ be a probability measure on $\R_+^2$ given by 
\begin{equation*}\label{eq:Fdu_finite_variance}
 F(\diff u)=f_1(u_1)f_2(u_2) \diff u_1 \diff u_2.
\end{equation*}
Furthermore, we define for $i=1,2$
\begin{equation}\label{eq:v_i}
 v_i \coloneqq \int_{\R_+} u_i^2 f_i(u_i) \diff u_i < \infty.
\end{equation}

The following result shows that the centred and renormalised random field on the space $\mathcal{M}_{P}$ converges to a centred Gaussian linear random field.
We emphasise that there does \emph{not} exist a diversity of regimes to distinguish in the finite variance case, which is also the much simpler case.
Nevertheless, the proof of this result can be viewed as a `prototype proof' for all other regimes.

\begin{thm}\label{thm:finite_variance}
Let $\lambda_\rho \to \infty$ as $\rho \to 0$.
Then, we have
$$
\frac{\widetilde{J}_\rho(\cdot)}{\rho^2 \sqrt{ \lambda_\rho v_1 v_2}} \xrightarrow{\mathcal{M}_{P}} X(\cdot)
$$
as $\rho \to 0$, where $v_i$ is defined in \eqref{eq:v_i} for $i=1,2$ and where $( X(\mu))_\mu$ is the centred Gaussian linear random field with covariance function
\begin{equation}\label{eq:cov_thm_fin_var}
 C_X(\mu,\nu)= \int_{\R^2} f_\mu(x) f_\nu(x) \diff x. 
\end{equation}
\end{thm}

\begin{rem}
We note that two limiting random fields in this paper have already arisen in identical form in related work.
The centred Gaussian linear random field with covariance function given in~\eqref{eq:cov_thm_fin_var} coincides with the corresponding one in the finite variance case of the random grain model where the size of a grain is given by a \emph{single} distribution (cf.\ Theorem~1 in~\cite{kaj2007}).
Moreover, the limiting random field consisting of integrals with respect to a stable random measure in the points scaling regime has also appeared there (cf.\ (13) in~\cite{kaj2007}).
The index of stability is given by the index of the regularly varying tail of the volume of a grain there and by the `heavier' tail index~$\gamma_1$ for the length of a box in our random boxes model.
All other limiting random fields seem to be new.
\end{rem}

\subsection{Statistical properties and extensions of the model}\label{sec:statprop}
In the following paragraphs, we give some statistical properties of the different scaling limits $Z$, $J_I$, $Y$, $J_L$, $S_{\gamma_1}$ and $X$.
We will omit the proofs of these facts because they can be verified easily.

\paragraph{Covariance.}
The covariance functions of the Gaussian random fields $Z$, $Y$ and $X$ are given in~\eqref{eq:cov_high_gaussian}, \eqref{eq:cov_low_gaussian} and \eqref{eq:cov_thm_fin_var}, respectively.
The covariance function of $J_I$ in the intermediate intensity regime is exactly the same as in the high intensity regime (see~\eqref{eq:cov_high_gaussian}), but the limit $J_I$ is not a Gaussian random field.
In the points scaling regime, the scaling limit $S_{\gamma_1}(\mu)$ is $\gamma_1$-stable and thus does not have a finite variance.
We distinguish two cases in the Poisson-lines scaling regime:
If $\gamma_2 < 2$, the compensated Poisson integral $J_L(\mu)$ does not have a finite variance.
In contrast, if we assume $\gamma_2 > 2$, i.e., the width of a box has a finite variance, the scaling limit $J_L(\mu)$ has a finite variance as well and the covariance function coincides with the one in the Gaussian-lines scaling regime (see~\eqref{eq:cov_low_gaussian}).

\paragraph{Translation invariance.}
Let $s \in \R^2$. We define the translation of a signed measure $\tau_s \mu$ by $\tau_s \mu (A) \coloneqq \mu(A-s)$ for any Borel set $A$.
We call a random field $W$ on $\mathcal{M}_W$ translation invariant if we have
$$
(W(\tau_s \mu))_{\mu \in \mathcal{M}_W} = (W(\mu))_{\mu \in \mathcal{M}_W}
$$
in finite-dimensional distributions for all $s \in \R^2$ ($\mathcal{M}_W$ has to be closed under translations $\tau_s$).
All limiting random fields $Z$, $J_I$, $Y$, $J_L$, $S_{\gamma_1}$ and $X$ are translation invariant on the respective spaces of signed measures.

\paragraph{Dilation.}
For all $a >0$ the dilation of a signed measure $\mu_a$ is given by $\mu_a (A) \coloneqq \mu(a^{-1}A)$ for any Borel set $A$.
We call a random field $W$ on $\mathcal{M}_W$ self-similar with index $H$ if we have
$$
(W(\mu_a))_{\mu \in \mathcal{M}_W} = (a^H W(\mu))_{\mu \in \mathcal{M}_W}
$$
in finite-dimensional distributions for all $a >0$ ($\mathcal{M}_W$ has to be closed under dilations $\mu_a$). 

The limiting Gaussian random fields $Z$, $Y$ and $X$ are self-similar with index $H=(2-\gamma_1 - \gamma_2)/{2}$, $H=-\gamma_1/{2}$ and $H=-1$, respectively.
In the points scaling regime, the limit $S_{\gamma_1}$ is self-similar with index $H=2 / \gamma_1 - 2$.
We emphasise that $H$ is negative in these cases.
If the reader expects $H$ to be positive, a reason may be found in the way of defining the dilation of a signed measure which, however, is common in literature.
One can also verify that the random field $J_I$ in the intermediate intensity regime is not self-similar (cf.\ \cite[p.\ 537]{kaj2007}).

One calls a random field $W$ with $\E W=0$ on $\mathcal{M}_W$ (which has to be again closed under dilation) aggregate-similar (cf.\ \cite{Bierme2010,Kaj_aggregate}) if there is a positive sequence $(a_m)_{m\geq 1}$ such that we have
$$
(W(\mu_{a_m}))_{\mu \in \mathcal{M}_W} = \left(\sum_{k=1}^m W^k(\mu)\right)_{\mu \in \mathcal{M}_W}
$$
in finite-dimensional distributions for all $m \geq 1$, where $(W^k)_{k \geq 1}$ are i.i.d.\ copies of $W$.

We obtain that the random fields $Z$, $Y$, $X$, $J_I$ and $S_{\gamma_1}$ are aggregate-similar with $a_m=m^{1/(2-\gamma_1-\gamma_2)}$, $a_m=m^{-1/\gamma_1}$, $a_m=m^{-1/2}$, $a_m=m^{1/(2-\gamma_1-\gamma_2)}$ and $a_m=m^{1/(2-2\gamma_1)}$, respectively.
Regarding the dilation in the Poisson-lines scaling regime, we mention that the scaling limit $J_L$ only fulfils a modification of aggregate-similarity, where the measure for the width is dilated simultaneously.

\vspace{\baselineskip}
Next, we sketch feasible extensions of our random boxes model.
For example, it is possible to allow non-negative $\sigma$-finite measures $F$ instead of restricting ourselves to probability measures or to consider boxes (hyper-rectangles) in~$\R^d$ with $d \geq 3$.
Moreover, the model can be extended as follows:

\paragraph{Randomly rotated boxes.}
A modification of the random boxes model consists in additionally endowing the rectangles with independent and uniformly distributed orientations.
We introduce the Haar measure $\diff \theta$ on the group of rotations $SO(2)$ in $\R^2$ and consider the Poisson random measure~$N^\circ_\rho$ on $\R^2 \times \R_+^2 \times SO(2)$
with intensity measure given by
$$
n^\circ_\rho(\diff x, \diff u,\diff \theta) = \lambda_\rho \diff x F_\rho(\diff u) \diff \theta.
$$
Then, the centred Poisson integral
$$
\widetilde{J}^\circ_\rho(\mu) \coloneqq \int_{\R^2 \times \R_+^2 \times SO(2)} \mu(B_{\theta}(x,u)) \widetilde{N}^\circ_\rho(\diff x,\diff u, \diff \theta)
$$
is the object of interest, where $B_{\theta}(0,u) \coloneqq \theta B(0,u)$ denotes the rectangle~$B(0,u)$ rotated by $\theta$ and $B_{\theta}(x,u)$ for $x \neq 0$ is defined by
$$
B_{\theta}(x,u) \coloneqq x + B_{\theta}(0,u).
$$

Since the probability measure $\diff \theta$ on the group $SO(2)$ is not affected by the scaling as $\rho \to 0$, one can proceed as in the proofs of Theorems~\ref{thm:high}, \ref{thm:inter}, \ref{thm:lowgaus}, \ref{thm:lowpois}, \ref{thm:lowpoint} and \ref{thm:finite_variance}.
One just has to change the spaces of signed measures slightly in order to obtain analogous (rotation invariant) limiting random fields for this modified random boxes model.
To keep the exposition comprehensible, we will not enter into more details in this paper.

\section{Preliminaries and technical tools}\label{sec:preliminaries}
First, we define the function $\Psi$ by
\begin{equation}\label{def:psi}
 \Psi(v) \coloneqq e^{iv}-1-iv, \quad \text{for $v \in \R$,}
\end{equation}
which we often require in order to represent characteristic functions.
Moreover, note that we use $c$ and $C$ from now on for constants which can differ from line to line as well as within a line.

\subsection{Spaces of signed measures}
We investigate the spaces of signed measures where the theorems of convergence in the high, intermediate and low intensity regimes hold, respectively.
The following proposition, which one can prove easily, ensures the linearity of these subspaces.

\begin{pro}
The subsets $\mathcal{M}^{\gamma_1,\gamma_2}$, $\mathcal{M}_{L}$ and $\mathcal{M}_{P}$ are linear subspaces of~$\mathcal{M}_2$.
\end{pro}

\begin{rem}
In Theorems~\ref{thm:high} and~\ref{thm:inter} in the high and intermediate intensity regimes, we additionally assume $\gamma_2<2$ instead of just $\gamma_2>\gamma_1$. The reason for that can be motivated in a natural way:
On the one hand, we have to require that there exists some $\alpha_2 > \gamma_2$ in Definition~\ref{def:M_high} in order to prove the theorems of convergence.
On the other hand, we want at least measures whose density functions have compact support to be contained in $\mathcal{M}^{\gamma_1,\gamma_2}$. As a consequence, $\alpha_2 \leq 2$ also has to be fulfilled.
Therefore, both inequalities can only be satisfied simultaneously for $\gamma_2<2$.
\end{rem}

\begin{rem}
We briefly comment on the characteristics of the spaces of signed measures in the low intensity sub-regimes (see Definitions~\ref{def:M_lowgaussian} and~\ref{def:M_lowpoint}).
The assumption that each signed measure has a density function is obviously necessary since the density function appears explicitly in the limiting random fields.
In contrast, we do not conjecture that the technical assumption on the decay of the density function in~\eqref{eq:decay_mu} is necessary as well.
Nevertheless, the reason for restricting the density functions to functions that decay at least exponentially fast is related to the maximal function of the signed measure given in \eqref{eq:m_mu^*} below.
We have to ensure that Lemma \ref{lem:m_mu}~(ii) below holds in order to prove the theorems of convergence.
\end{rem}

Next, we briefly touch on the comparison of these spaces of signed measures for $\gamma_i \in (1,2)$ for $i=1,2$.
We observe that the space $\mathcal{M}^{\gamma_1,\gamma_2}$ contains measures which do not have to have a density.
Therefore, there exist some $\mu \in \mathcal{M}^{\gamma_1,\gamma_2}$, but $\mu \notin  \mathcal{M}_{k}$ for $k \in \{L,P\}$.
Conversely, we obtain the following result:
\begin{pro}\label{lem:M_L_in_M_gamma}
Let $\gamma_i \in (1,2)$ for $i=1,2$. We have $\mathcal{M}_{k} \subseteq \mathcal{M}^{\gamma_1,\gamma_2}$
for $k \in \{L,P\}$.
\end{pro}
\begin{proof}[Sketch of proof]
Note that the density function of a signed measure in $\mathcal{M}_{k}$ for $k \in \{L,P\}$ satisfies
$$
\vert f_\mu(x) \vert \leq C_\mu e^{-c_\mu(\vert x_1 \vert + \vert x_2 \vert)}
$$
for all $x \in \R^2$ for some $C_\mu>0$ and $c_\mu>0$.
One can compute that
$$
\int_{\R} \left( \int_{\left[x_1-\frac{u_1}{2}, x_1+\frac{u_1}{2}\right]} e^{-c_\mu \vert y_1 \vert} \diff y_1 \right)^2 \diff x_1 \leq c \min\left(u_1,u_1^{2}\right)
$$
by a case distinction for $u_1 \leq 1$ and $u_1 > 1$.
Using the product form, the validity of inequality~\eqref{eq:phibound} follows.
\end{proof}

\subsection{Existence of the random fields}
We deal with the existence of the random field $\widetilde{J}$ of interest and all the limiting random fields in the different scaling regimes.
Using Lemma~12.13 in \cite{kallenberg2002foundations}, we can verify that the random fields $J$ and $\widetilde{J}$ exist because we have
$$
\int_{\R^2 \times \R_+^2} \vert \mu(B(x,u)) \vert n(\diff x,\diff u) \leq  \lambda \Vert \mu \Vert \int_{\R_+^2} u_1 u_2 F(\diff u)< \infty.
$$

Furthermore, by standard facts on Poisson integrals and Fubini's theorem, the expected value of $J(\mu)$ is finite and given by
$$
\E J(\mu) = \lambda  \mu (\R^2) \int_{\R_+} u_1 f_1(u_1) \diff u_1 \int_{\R_+} u_2 f_2(u_2) \diff u_2.
$$
Using the function $\Psi$ defined in~\eqref{def:psi}, the characteristic function of $\widetilde{J}(\mu)$ is given by
$$
\E \left(e^{i\widetilde{J}(\mu)}\right) = \exp \left( \int_{\R_+^2} \int_{\R^2} \Psi(\mu(B(x,u))) \lambda \diff x F(\diff u) \right).
$$

\begin{lem}\label{lem:intphifinite}
We have for all $\mu \in \mathcal{M}^{\gamma_1,\gamma_2}$
$$
\int_{\R^2 \times \R_+^2} \mu(B(x,u))^2 \frac{1}{u_1^{\gamma_1+1}} \frac{1}{u_2^{\gamma_2+1}} \diff (x,u) < \infty.
$$
\end{lem}
\begin{proof}[Proof]
This follows directly from Definition~\ref{def:M_high} of the space $\mathcal{M}^{\gamma_1,\gamma_2}$ by using the estimate in~\eqref{eq:phibound} for the function $\phi$ defined by
\begin{equation}\label{eq:phi}
 \phi(u) \coloneqq \int_{\R^2} \mu(B(x,u))^2 \diff x, \quad \text{for $u \in \R_+^2$.}
\end{equation}
\end{proof}

In the following, we briefly note that all the limiting random fields obtained in Theorems~\ref{thm:high}, \ref{thm:inter}, \ref{thm:lowgaus}, \ref{thm:lowpois}, \ref{thm:lowpoint} and \ref{thm:finite_variance} are well-defined:
\begin{itemize}
\item Using Lemma~\ref{lem:intphifinite}, one can check easily that the right hand side of \eqref{eq:cov_high_gaussian} is a symmetric, positive-semidefinite function such that there is a centred Gaussian linear random field $Z$ with covariance function given by \eqref{eq:cov_high_gaussian}.
\item The existence of $J_I$ follows from Lemma~12.13 in \cite{kallenberg2002foundations} and Lemma~\ref{lem:intphifinite}.
\item The proof of Theorem~\ref{thm:lowgaus} shows that $\sigma^2$ given in \eqref{eq:sigma2lowgaus} is finite.
Hence, it can serve to construct the covariance function of a centred Gaussian linear random field $Y$.
\item The existence of the compensated Poisson integral $J_L(\mu)$ for $\mu \in \mathcal{M}_{L}$ given in \eqref{eq:poissonlineslimit} can be verified by Lemma~12.13 in \cite{kallenberg2002foundations}.
One just has to show
$$
\int_{\R^2 \times \R_+^2} \min \left( \left\vert g(x,u) \right\vert, g(x,u) ^2 \right)\frac{1}{u_1^{\gamma_1+1}} f_2 \left( u_2 \right) \diff(x,u) < \infty,
$$
where
$$
g(x,u) \coloneqq u_2 \int_{\left[x_1-\frac{u_1}{2}, x_1+\frac{u_1}{2}\right]} f_\mu(y_1,x_2)\diff y_1.
$$
This can be seen by a case distinction.
Let us start with the following general consideration.
There is an $\eps >0$ such that
\begin{equation}\label{eq:low_estimate^12}
 \min \left( \vert v \vert ,  v^2 \right) \leq \min \left( \vert v \vert^{\gamma_1-\eps} , \vert v \vert^{\gamma_1+\eps} \right)
\end{equation}
with $1<\gamma_1 - \eps $ and \mbox{$\gamma_1 + \eps < \min \left( \gamma_2, 2 \right)$}.
Furthermore, we observe
\begin{equation}\label{eq:low_estimate2}
 \begin{split}
 & \min \left( \vert ab \vert^{\gamma_1-\eps} , \vert ab \vert^{\gamma_1+\eps} \right) \\
 \leq & \min \left( \vert a \vert^{\gamma_1-\eps} (\vert b \vert^{\gamma_1-\eps}+\vert b \vert^{\gamma_1+\eps}) , \vert a \vert^{\gamma_1+\eps} (\vert b \vert^{\gamma_1+\eps}+\vert b \vert^{\gamma_1-\eps}) \right) \\
 = & \min \left( \vert a \vert^{\gamma_1-\eps} , \vert a \vert^{\gamma_1+\eps} \right) (\vert b \vert^{\gamma_1-\eps}+\vert b \vert^{\gamma_1+\eps}).
 \end{split}
\end{equation}
We use \eqref{eq:low_estimate^12}, \eqref{eq:low_estimate2} and the assumption~\eqref{eq:decay_mu} from Definition~\ref{def:M_lowgaussian} to obtain
\begin{equation}\label{eq:exponent_gamma_1+eps}
 \begin{split}
 &  \min \left( \left\vert g(x,u) \right\vert,   g(x,u) ^2 \right) \\
 \leq  & \min \left( \left(C_\mu g_1(x_1,u_1) e^{-c_\mu \vert x_2 \vert} \right)^{\gamma_1-\eps} ,  \left(C_\mu g_1(x_1,u_1) e^{-c_\mu \vert x_2 \vert}\right)^{\gamma_1+\eps} \right) \\
 & \times (\vert u_2 \vert^{\gamma_1-\eps}+\vert u_2 \vert^{\gamma_1+\eps}), \\
 \end{split}
\end{equation}
where
\begin{equation}\label{eq:g_1}
 g_1(x_1,u_1) \coloneqq \int_{\left[x_1-\frac{u_1}{2}, x_1+\frac{u_1}{2}\right]} e^{-c_\mu \vert y_1 \vert } \diff y_1. 
\end{equation}
Since 
$$
\int_{\R_+} (\vert u_2 \vert^{\gamma_1-\eps}+\vert u_2 \vert^{\gamma_1+\eps}) f_2 \left( u_2 \right) \diff u_2  < \infty
$$
due to $\gamma_1 + \eps < \gamma_2$ and the asymptotic behaviour of $f_2$, and since
$$
\int_{\R} e^{-c_\mu  \vert x_2  \vert({\gamma_1 \pm \eps})} \diff x_2 < \infty,
$$
it remains to show that
\begin{equation}\label{eq:exist_lowpoiss}
 \int_{\R \times \R_+} \min \left( g_1(x_1,u_1)^{\gamma_1-\eps} ,  g_1(x_1,u_1) ^{\gamma_1+\eps} \right) \frac{1}{u_1^{\gamma_1+1}} \diff(x_1,u_1) 
\end{equation}
is finite.
For $u_1 \leq 1$, we obtain
\begin{equation*}
 \begin{split}
 & \int_{\R } g_1(x_1,u_1) ^{\gamma_1+\eps} \diff x_1 \\
 \leq & \, 2 \int_{\R_+ } \left(  \int_{\left[x_1-\frac{u_1}{2}, x_1+\frac{u_1}{2}\right]}  e^{-c_\mu  y_1} \diff y_1 \right)^{\gamma_1+\eps} \diff x_1 \\
 \leq & \,  2 \int_{\R_+ } \left(  \int_{\left[x_1-\frac{u_1}{2}, x_1+\frac{u_1}{2}\right]}  e^{-c_\mu \left( x_1-\frac{u_1}{2}\right) } \diff y_1 \right)^{\gamma_1+\eps} \diff x_1 \\ 
 =& \,  2  e^{c_\mu   \frac{u_1}{2}({\gamma_1+\eps})} u_1^{\gamma_1+\eps}\int_{\R_+ } e^{-c_\mu ({\gamma_1+\eps}) x_1}\diff x_1   \leq   C  u_1^{\gamma_1+\eps}.  \\ 
 \end{split}
\end{equation*}
In the case of $u_1 \geq 1$, we observe
\begin{equation*}
 \begin{split}
 & \int_{\R } g_1(x_1,u_1) ^{\gamma_1-\eps} \diff x_1 \\
 \leq & \, \int_{\left[-\frac{u_1}{2}, \frac{u_1}{2}\right]} \left( \int_{\R} e^{-c_\mu \vert y_1 \vert} \diff y_1 \right)^{\gamma_1-\eps} \diff x_1 + 2 \int_{\left(\frac{u_1}{2}, \infty \right)}  g_1(x_1,u_1) ^{\gamma_1-\eps} \diff x_1  \\
 = & \left( \frac{2}{c_\mu} \right) ^{\gamma_1-\eps} u_1 + \frac{2}{c_\mu^{\gamma_1-\eps}} \int_{\left(\frac{u_1}{2}, \infty \right)} \left( 
 e^{-c_\mu x_1} \left( e^{c_\mu \frac{u_1}{2}} - e^{-c_\mu \frac{u_1}{2}} \right) \right)^{\gamma_1-\eps} \diff x_1  \\
 \leq & \, C \left( u_1 + \left( e^{c_\mu \frac{u_1}{2}} - e^{-c_\mu \frac{u_1}{2}}  \right)^{\gamma_1-\eps} \frac{1}{c_\mu({\gamma_1-\eps})} e^{-c_\mu ({\gamma_1-\eps}) \frac{u_1}{2}}\right)\\
 \leq  & \, C \left( u_1 +
 \left( e^{c_\mu \frac{u_1}{2}} \right)^{\gamma_1-\eps} e^{-c_\mu ({\gamma_1-\eps}) \frac{u_1}{2}}\right) \leq C \left( u_1 +1 \right) \leq C u_1.\\
\end{split}
\end{equation*}
Finally, we can split the integral in \eqref{eq:exist_lowpoiss} into two parts following this case distinction and see that these are bounded by
$$
\int_{\left( 0,1\right]} C u_1^{\gamma_1+\eps} \frac{1}{u_1^{\gamma_1+1}} \diff u_1 < \infty \quad \text{and} \quad \int_{\left( 1, \infty \right)} C u_1 \frac{1}{u_1^{\gamma_1+1}} \diff u_1 < \infty,
$$
respectively.
Therefore, the existence of the compensated Poisson integral $J_L(\mu)$ for $\mu \in \mathcal{M}_{L}$ is proven since the integral in \eqref{eq:exist_lowpoiss} is finite.
We note that in inequality~\eqref{eq:exponent_gamma_1+eps} the particular exponent $\gamma_1-\eps$  is not required for this proof and one could also replace $\gamma_1-\eps$ by 1. However, we stick to the exponent $\gamma_1-\eps$ because we will need the estimates here for later purposes, for instance, in the proof of Theorem~\ref{thm:lowpois}.
\item Since $f_\mu \in L^{\gamma_1}(\R^2)$ for $\mu \in \mathcal{M}_{P}$, the random linear functional $S_{\gamma_1}(\mu)$ given in \eqref{eq:def_stable_int} is well-defined.
We refer to Chapter~3 in \cite{samorodnitsky2000stable} for an extensive discussion.
\item We can deduce from the proof of Theorem~\ref{thm:finite_variance} that the integral in~\eqref{eq:cov_thm_fin_var} is finite and serves to construct the covariance function of a centred Gaussian linear random field $X$.
\end{itemize} 

\subsection{Further useful lemmas}
We continue with some useful lemmas that we use in the proofs of the main results in Section~\ref{sec:proofmain}.

\begin{lem}\label{lem:intasympf}
Let $F$ be a measure on $\R_+^2$ according to \eqref{eq:Fdu} and to the asymptotic behaviour specified there.
Furthermore, let $g$ be a continuous function on $\R_+^2$ such that there is a constant $C>0$ for some $\alpha_i>\gamma_i$ for $i=1,2$ such that
\begin{equation}\label{eq:gboundf}
 \vert g(u) \vert \leq C \min\left(u_1,u_1^{\alpha_1} \right) \min\left(u_2,u_2^{\alpha_2} \right)
\end{equation}
for all $u \in \R_+^2$.
Then, we have as $\rho \to 0$
$$
\int_{\R_+^2} g(u) F_\rho(\diff u) \sim \rho^{\gamma_1+\gamma_2} \int_{\R_+^2} g(u) \frac{1}{u_1^{\gamma_1+1}} \frac{1}{u_2^{\gamma_2+1}} \diff u.
$$
\end{lem}
\begin{proof}[Proof]
The idea of the proof is to split the integral \mbox{$\int_{\R_+^2} g(u) F_\rho(\diff u)$} into four parts and treat the four integrals separately.

Let $\eps>0$ be given and define the constant $c_0$ by
\begin{equation}\label{eq:def_c}
 c_0\int_{\R_+^2} \vert g(u) \vert \frac{1}{u_1^{\gamma_1+1}} \frac{1}{u_2^{\gamma_2+1}} \diff u =  \left\vert \int_{\R_+^2}g(u) \frac{1}{u_1^{\gamma_1+1}} \frac{1}{u_2^{\gamma_2+1}} \diff u \right\vert.
\end{equation}
(We note that $ \int_{\R_+^2} \vert g(u) \vert \frac{1}{u_1^{\gamma_1+1}} \frac{1}{u_2^{\gamma_2+1}} \diff u < \infty $ because of inequality~\eqref{eq:gboundf} and that 
one has to treat the special case with $\int_{\R_+^2}g(u) \frac{1}{u_1^{\gamma_1+1}} \frac{1}{u_2^{\gamma_2+1}} \diff u = 0$ slightly differently.)
Choose $N=N(\eps)$ such that for all $u_i>N$ for $i=1,2$ we have
\begin{equation}\label{eq:N_1}
 f_i(u_i) \leq \frac{2}{u_i^{\gamma_i+1}} 
\end{equation}
and
\begin{equation}\label{eq:N_12}
 \left\vert f_1(u_1)f_2(u_2) - \frac{1}{u_1^{\gamma_1+1}} \frac{1}{u_2^{\gamma_2+1}} \right\vert \leq c_0 \frac{\eps}{8} \frac{1}{u_1^{\gamma_1+1}} \frac{1}{u_2^{\gamma_2+1}},
\end{equation}
which is feasible due to the power-law assumption on the measure $F$. 
We write $\R_+^2 = \bigcup_{k=1}^4 \Omega_k$ with 
\begin{equation}\label{eq:subdomains}
 \begin{split}
  \Omega_1 & \coloneqq (\rho N, \infty)^2, \\
  \Omega_2 & \coloneqq (0, \rho N]^2, \\
  \Omega_3 &\coloneqq (\rho N, \infty) \times (0, \rho N],\\
  \Omega_4 &\coloneqq  (0, \rho N] \times (\rho N, \infty).
 \end{split}
\end{equation}
From now on, we discuss the four corresponding integrals separately. 
\begin{enumerate}
 \item[1.)] Using \eqref{eq:N_12}, we get
 \begin{align} 
 &  \left\vert \int_{\Omega_1} g(u) F_\rho(\diff u) - \rho^{\gamma_1+\gamma_2} \int_{\R_+^2} g(u) \frac{1}{u_1^{\gamma_1+1}} \frac{1}{u_2^{\gamma_2+1}} \diff u \right\vert  \nonumber\\
 \leq &   \int_{\Omega_1} \vert g(u) \vert \left\vert f_1\left(\frac{u_1}{\rho}\right) \frac{1}{\rho}  f_2\left(\frac{u_2}{\rho}\right)\frac{1}{\rho} - \rho^{\gamma_1+\gamma_2} \frac{1}{u_1^{\gamma_1+1}} \frac{1}{u_2^{\gamma_2+1}} \right\vert \diff u  \nonumber\\
 & \quad +  \rho^{\gamma_1+\gamma_2} \int_{\R_+^2 \setminus \Omega_1} \vert g(u) \vert \frac{ 1}{u_1^{\gamma_1+1}} \frac{1}{u_2^{\gamma_2+1}} \diff u \nonumber \\
 \leq & \, c_0 \frac{\eps}{8}  \rho^{\gamma_1+\gamma_2} \int_{\R_+^2} \vert g(u) \vert  \frac{1}{u_1^{\gamma_1+1}} \frac{1}{u_2^{\gamma_2+1}} \diff u \nonumber \\
 & \quad +  \rho^{\gamma_1+\gamma_2} \int_{\R_+^2 \setminus \Omega_1} \vert g(u) \vert \frac{ 1}{u_1^{\gamma_1+1}} \frac{1}{u_2^{\gamma_2+1}} \diff u \label{eq:int_conv_0} \\
 \leq & \, c_0 \frac{\eps}{4}   \rho^{\gamma_1+\gamma_2} \int_{\R_+^2} \vert g(u) \vert  \frac{1}{u_1^{\gamma_1+1}} \frac{1}{u_2^{\gamma_2+1}} \diff u \nonumber 
 \end{align}
 for $\rho$ small enough, where we also used that the integral in \eqref{eq:int_conv_0} converges to zero by the dominated convergence theorem.
 Hence, we can deduce together with the definition of $c_0$ in \eqref{eq:def_c} that there exists some $\rho_1 > 0$ such that for all $\rho < \rho_1$ we get
 \begin{equation*}
  \begin{split}
  & \left\vert\frac{ \int_{\Omega_1} g(u) F_\rho(\diff u)}{\rho^{\gamma_1+\gamma_2} \int_{\R_+^2} g(u) \frac{1}{u_1^{\gamma_1+1}} \frac{1}{u_2^{\gamma_2+1}} \diff u}   -1\right\vert   \\
  = &  \,\frac{\left\vert \int_{\Omega_1} g(u) F_\rho(\diff u) -\rho^{\gamma_1+\gamma_2} \int_{\R_+^2} g(u) \frac{1}{u_1^{\gamma_1+1}} \frac{1}{u_2^{\gamma_2+1}}\diff u \right\vert}{\rho^{\gamma_1+\gamma_2}\left\vert\int_{\R_+^2} g(u) \frac{1}{u_1^{\gamma_1+1}} \frac{1}{u_2^{\gamma_2+1}} \diff u \right\vert} \\ 
  \leq & \,\frac{c_0 \frac{\eps}{4}  \int_{\R_+^2} \vert g(u) \vert  \frac{1}{u_1^{\gamma_1+1}} \frac{1}{u_2^{\gamma_2+1}} \diff u}{\left\vert\int_{\R_+^2} g(u) \frac{1}{u_1^{\gamma_1+1}} \frac{1}{u_2^{\gamma_2+1}} \diff u \right\vert} \leq \frac{\eps}{4}.
  \end{split}
 \end{equation*}
 \item[2.)] We can show $\left\vert \int_{\Omega_2} g(u) F_\rho(\diff u) \right\vert \in o(\rho^{\gamma_1+\gamma_2}).$
 Indeed, using \eqref{eq:gboundf}, we obtain 
 \begin{equation*}
  \begin{split}
  \left\vert \int_{\Omega_2} g(u) F_\rho(\diff u) \right\vert 
  \leq & \, C \int_0^{\rho N} \int_0^{\rho N} u_1^{\alpha_1} u_2^{\alpha_2} f_1\left(\frac{u_1}{\rho} \right) f_2\left(\frac{u_2}{\rho}\right)\frac{1}{\rho^2} \diff u_1  \diff u_2  \\
  = & \, C \rho^{\alpha_1+\alpha_2} \int_0^N \int_0^N u_1^{\alpha_1} u_2^{\alpha_2} f_1(u_1)f_2(u_2) \diff u_1 \diff u_2   \\
  \leq & \, C \rho^{\alpha_1+\alpha_2} N^{\alpha_1+\alpha_2}.
  \end{split}
 \end{equation*} 
 Since $\alpha_1+\alpha_2 > \gamma_1+\gamma_2$, the assertion is true for $\rho \to 0$.
 More precisely, for $\eps$ and $N$ as above there exists some $\rho_2>0$ such that for all $\rho<\rho_2$ we have
 $$
 \left\vert \frac{\int_{\Omega_2} g(u) F_\rho(\diff u)}{\rho^{\gamma_1+\gamma_2} \int_{\R_+^2} g(u) \frac{1}{u_1^{\gamma_1+1}} \frac{1}{u_2^{\gamma_2+1}} \diff u} \right\vert < \frac{\eps}{4}.
 $$
 \item[3.)] We show $\left\vert \int_{\Omega_3} g(u) F_\rho(\diff u) \right\vert \in o(\rho^{\gamma_1+\gamma_2}).$
 We obtain for ${N}$ satisfying~\eqref{eq:N_1}
 \begin{equation*}
 \begin{split}
  & \left\vert \int_{\Omega_3} g(u) F_\rho(\diff u) \right\vert \\
  \leq & \,  \int_{\rho N}^\infty \int_0^{\rho N} \vert g(u) \vert f_2\left(\frac{u_2}{\rho}\right)\frac{1}{\rho} \diff u_2  f_1 \left(\frac{u_1}{\rho}\right) \frac{1}{\rho}  \diff u_1 \\
  \leq & \, C  \int_{\rho N}^\infty \int_0^{\rho N}  \min\left(u_1,u_1^{\alpha_1} \right) \min\left(u_2,u_2^{\alpha_2} \right)  f_2 \left(\frac{u_2}{\rho}\right)\frac{1}{\rho} \diff u_2   \frac{\rho^{\gamma_1}}{u_1^{\gamma_1+1}} \diff u_1 \\
  \leq & \, C     \rho^{\gamma_1} \int_{\rho N}^\infty \min\left(u_1,u_1^{\alpha_1}\right)  \frac{1}{u_1^{\gamma_1+1}} \diff u_1  \int_0^{\rho N} u_2^{\alpha_2} f_2\left(\frac{u_2}{\rho}\right)\frac{1}{\rho} \diff u_2  \\
  = & \, C    \rho^{\gamma_1} \rho^{\alpha_2}  \int_0^{N} u_2^{\alpha_2}  f_2(u_2) \diff u_2  \leq C  \rho^{\gamma_1+\alpha_2} N^{\alpha_2}.\\
 \end{split}
 \end{equation*} 
 Since $\gamma_1+\alpha_2 > \gamma_1+\gamma_2$, we are done.
 In other words, for $\eps$ and $N$ as above, there exists some $\rho_3>0$ such that for all $\rho<\rho_3$ we have
 $$
 \left\vert \frac{\int_{\Omega_3} g(u) F_\rho(\diff u)}{\rho^{\gamma_1+\gamma_2} \int_{\R_+^2} g(u) \frac{1}{u_1^{\gamma_1+1}} \frac{1}{u_2^{\gamma_2+1}} \diff u} \right\vert < \frac{\eps}{4}.
 $$
 \item[4.)]
 Proceeding analogously to 3.), one shows $\left\vert \int_{\Omega_4} g(u) F_\rho(\diff u) \right\vert \in o(\rho^{\gamma_1+\gamma_2}).$
 Again, for $\eps$ and $N$ as above, there exists some $\rho_4>0$ such that for all $\rho<\rho_4$ we obtain
 $$
 \left\vert \frac{\int_{\Omega_4} g(u) F_\rho(\diff u)}{\rho^{\gamma_1+\gamma_2} \int_{\R_+^2} g(u) \frac{1}{u_1^{\gamma_1+1}} \frac{1}{u_2^{\gamma_2+1}} \diff u} \right\vert < \frac{\eps}{4}.
 $$
\end{enumerate}
Finally, we are able to deduce the assertion of the lemma:
We just define \mbox{$\rho_0 \coloneqq \min_{k \in \{1,\ldots,4\}} \rho_k $}.
Then, we obtain for all $\rho<\rho_0$, by splitting the domain of integration as mentioned above,
\begin{equation*}
 \begin{split}
 & \left\vert  \frac{\int_{\R_+^2} g(u) F_\rho(\diff u)}{\rho^{\gamma_1+\gamma_2} \int_{\R_+^2} g(u) \frac{1}{u_1^{\gamma_1+1}} \frac{1}{u_2^{\gamma_2+1}} \diff u} -1\right\vert\\
 \leq &  \left\vert  \frac{\int_{\Omega_1} g(u) F_\rho(\diff u)}{\rho^{\gamma_1+\gamma_2} \int_{\R_+^2} g(u) \frac{1}{u_1^{\gamma_1+1}} \frac{1}{u_2^{\gamma_2+1}} \diff u}-1 \right\vert + \sum_{k=2}^4  \left\vert \frac{\int_{\Omega_k} g(u) F_\rho(\diff u)}{\rho^{\gamma_1+\gamma_2} \int_{\R_+^2} g(u) \frac{1}{u_1^{\gamma_1+1}} \frac{1}{u_2^{\gamma_2+1}} \diff u} \right\vert \\
 \leq & \, \frac{\eps}{4} + \sum_{k=2}^4 \frac{\eps}{4}  = \eps,
 \end{split}
\end{equation*}
where we used the results from the four parts above.
\end{proof}

\begin{lem}\label{lem:intasymp2f}
Let $F$ be a measure on $\R_+^2$ according to \eqref{eq:Fdu} and to the asymptotic behaviour specified there.
Furthermore, let $\left( g_\rho \right)$ be a family of continuous functions on $\R_+^2$ with
$$
\lim_{\rho \to 0} \rho^{\gamma_1 + \gamma_2} g_\rho(u) = 0
$$
for all $u \in \R_+^2$ and
$$
\rho^{\gamma_1 + \gamma_2} \vert g_\rho(u) \vert \leq C \min\left(u_1,u_1^{\alpha_1} \right) \min\left(u_2,u_2^{\alpha_2} \right)
$$
for some constants $C>0$ and $\alpha_i>\gamma_i$ for $i=1,2$ for all $u \in \R_+^2$.
Then, we have
\begin{equation}\label{eq:lem_to_0}
 \lim_{\rho \to 0} \int_{\R_+^2} g_\rho(u) F_\rho(\diff u) = 0.
\end{equation}
\end{lem}
\begin{proof}[Proof]
The assumptions on $g_\rho$ ensure that for all $\rho >0$
$$
\int_{\R_+^2} \rho^{\gamma_1+\gamma_2} \vert g_\rho(u) \vert \frac{1}{u_1^{\gamma_1+1}} \frac{1}{u_2^{\gamma_2+1}} \diff u < \infty,
$$
that there is an integrable majorant and that we get
\begin{equation}\label{eq:dom_conv_to0}
 \lim_{\rho \to 0} \int_{\R_+^2} \rho^{\gamma_1+\gamma_2} \vert g_\rho(u) \vert \frac{1}{u_1^{\gamma_1+1}} \frac{1}{u_2^{\gamma_2+1}} \diff u = 0 
\end{equation}
by the dominated convergence theorem.

Due to the power-law assumption on $F$, we can choose $N>0$ such that for all $u_i>N$ for $i=1,2$ we have
\begin{equation}\label{eq:N_1Lem2}
 f_i(u_i) \leq \frac{2}{u_i^{\gamma_i+1}}.
\end{equation}
We use the same definition of the domains $\Omega_k$ for $k=1,\ldots,4$ as in~\eqref{eq:subdomains} and continue discussing the corresponding four integrals separately. First, using~\eqref{eq:N_1Lem2} we get
\begin{equation*}
 \begin{split}
 \left\vert \int_{\Omega_1} g_\rho(u) F_\rho(\diff u) \right\vert 
 \leq & \, \int_{\rho N}^\infty \int_{\rho N}^\infty \vert  g_\rho(u) \vert f_1 \left(\frac{u_1}{\rho}\right) \frac{1}{\rho} f_2\left(\frac{u_2}{\rho}\right)\frac{1}{\rho}     \diff u_1 \diff u_2\\
 \leq & \,  \int_{0}^\infty \int_{0}^\infty \rho^{\gamma_1+\gamma_2} \vert  g_\rho(u) \vert\frac{2}{u_1^{\gamma_1+1}} \frac{2}{u_2^{\gamma_2+1}}  \diff u_1 \diff u_2.\\
 \end{split}
\end{equation*} 
Therefore, we obtain together with \eqref{eq:dom_conv_to0} that
$$
\lim_{\rho \to 0} \int_{\Omega_1} g_\rho(u) F_\rho(\diff u) =0.
$$
Using the second assumption on $g_\rho$ and \eqref{eq:N_1Lem2}, one can check that
$$
\left\vert \int_{\Omega_k} g_\rho(u) F_\rho(\diff u) \right\vert \to 0
$$
as $\rho \to 0$ for $k=2,3,4$ by proceeding analogously to the corresponding parts in the proof of Lemma~\ref{lem:intasympf}.
Combining all four partial results, we can deduce~\eqref{eq:lem_to_0}.
\end{proof}

We introduce for a signed measure $\mu \in \mathcal{M}_{k}$ for $k \in \{L,P\}$ the local averages $m_\mu(x,u)$  by
\begin{equation}\label{eq:m_mu}
 m_\mu(x,u) \coloneqq \frac{1}{u_1 u_2} \int_{B(x,u)} f_\mu(y) \diff y
\end{equation}
and the maximal function $m_\mu^*$ by
\begin{equation}\label{eq:m_mu^*}
 m_\mu^*(x) \coloneqq \sup_{u \in \R^2_+}  \frac{1}{u_1 u_2} \int_{B(x,u)} \vert f_\mu(y) \vert \diff y.
\end{equation} 

\begin{lem}\label{lem:m_mu}
Let $n_i(\rho) \to 0$ as $\rho \to 0$ for $i=1,2$.
\begin{enumerate}
 \item \label{lem:(i)} For $\mu \in \mathcal{M}_{P}$, we have 
 $$
 \lim_{\rho \to 0} m_\mu \left( x, \begin{mysmallmatrix} n_1(\rho) u_1 \\ n_2(\rho) u_2 \end{mysmallmatrix}\right)= f_\mu(x), \quad \text{for all $(x,u) \in \R^2 \times \R_+^2$.}
 $$
 \item \label{lem:(ii)} Let $\beta>1$.
 For $\mu \in \mathcal{M}_{k}$ for $k \in \{L,P\}$, there is a function $g \in L^\beta(\R^2)$ such that $m_\mu^*(x)  \leq g(x)$ for all $x \in \R^2$.
\end{enumerate}
\end{lem}
\begin{proof}[Proof]
(i) The assertion is true because the function $f_\mu$ is continuous and because there exists for all $\delta>0$ some $\rho_0>0$ small enough such that the set $B \left( x, \begin{mysmallmatrix} n_1(\rho) u_1 \\ n_2(\rho) u_2 \end{mysmallmatrix}\right)$ is contained in the $\ell^\infty$-ball with centre $x$ and radius $\delta$ for all $\rho<\rho_0$.
\noindent
(ii) We only require the assumption~\eqref{eq:decay_mu} on $\mu \in \mathcal{M}_{k}$ for $k \in \{L,P\}$.
We obtain
\begin{align}\label{eq:sup_1d}
 m_\mu^*(x) \leq & \, C_\mu \sup_{u \in \R^2_+}  \frac{1}{u_1 u_2} \int_{B(x,u)}  e^{-c_\mu \vert y_1 \vert} e^{-c_\mu \vert y_2 \vert} \diff y \nonumber \\
 = & \,C_\mu  \prod_{i=1,2} \sup_{u_i \in \R_+}  \frac{1}{u_i} \int_{\left[x_i-\frac{u_i}{2}, x_i+\frac{u_i}{2}\right]}  e^{-c_\mu\vert y_i \vert} \diff y_i
\end{align}
and study the supremum in \eqref{eq:sup_1d} by a case distinction.
Let $x_i>0$.
We estimate
\begin{equation*}
 \begin{split}
 & \, \sup_{u_i >0}  \frac{1}{u_i} \int_{\left[x_i-\frac{u_i}{2}, x_i+\frac{u_i}{2}\right]}  e^{-c_\mu\vert y_i \vert} \diff y_i \\
 \leq & \, \sup_{0< \frac{u_i}{2} \leq x_i}  \frac{1}{u_i} \int_{\left[x_i-\frac{u_i}{2}, x_i+\frac{u_i}{2}\right]}  e^{-c_\mu\vert y_i \vert} \diff y_i + \sup_{\frac{u_i}{2} \geq x_i}  \frac{1}{u_i} \int_{\left[x_i-\frac{u_i}{2}, x_i+\frac{u_i}{2}\right]}  e^{-c_\mu\vert y_i \vert} \diff y_i
 \end{split}
\end{equation*}
and treat the two terms in the last line separately.
For $0< \frac{u_i}{2}\leq x_i$, we get
\begin{align}
 \frac{1}{u_i} \int_{\left[x_i-\frac{u_i}{2}, x_i+\frac{u_i}{2}\right]}  e^{-c_\mu\vert y_i \vert} \diff y_i = & \, \frac{1}{u_i} \frac{1}{c_\mu}  \left(  e^{-c_\mu\left(x_i - \frac{u_i}{2}\right)} -  e^{-c_\mu\left (x_i + \frac{u_i}{2}\right)} \right) \nonumber\\
 = & \, \frac{e^{-c_\mu x_i}}{c_\mu} \frac{e^{\frac{c_\mu u_i}{2}} -  e^{-\frac{c_\mu u_i}{2}}}{u_i} \nonumber  \\
 \leq & \, \frac{e^{-c_\mu x_i}}{c} \frac{e^{c_\mu x_i} -  e^{-c_\mu x_i}}{2x_i}  \label{eq:mon}\\
 \leq & \, \frac{1}{2c_\mu x_i}, \nonumber 
\end{align}
where we used the fact that the function
$$
h(u_i) \coloneqq  \frac{e^{c u_i} -  e^{-cu_i}}{u_i}
$$
is increasing for $u_i\geq 0$ in \eqref{eq:mon}.
This can be seen by
\begin{equation*}
 \begin{split}
 h(u_i) = & \, \frac{1}{u_i} \left( \sum_{k=0}^\infty \frac{(cu_i)^k}{k!}-\sum_{k=0}^\infty \frac{(-cu_i)^k}{k!} \right)   =  \frac{1}{u_i}  \sum_{l=0}^\infty \frac{2(cu_i)^{2l+1}}{(2l+1)!} = 2   \sum_{l=0}^\infty \frac{(cu_i)^{2l}}{(2l+1)!} \\
 \end{split}
\end{equation*}
because the last term is increasing in $u_i$.
For $\frac{u_i}{2} \geq x_i$, we observe
$$
\frac{1}{u_i} \int_{\left[x_i-\frac{u_i}{2}, x_i+\frac{u_i}{2}\right]}  e^{-c_\mu \vert y_i \vert} \diff y_i \leq  \frac{1}{2x_i}  \int_{\R}  e^{-c_\mu \vert y_i \vert} \diff y_i \leq  \frac{1}{c_\mu x_i}.
$$
Combining the estimates, we get
$$
\sup_{u_i >0}  \frac{1}{u_i} \int_{\left[x_i-\frac{u_i}{2}, x_i+\frac{u_i}{2}\right]}  e^{-c_\mu \vert y_i \vert} \diff y_i \leq  \frac{2}{c_\mu x_i}.
$$
The corresponding estimate with $\vert x_i \vert$ for $x_i<0$ follows directly because of symmetry.
Furthermore, we can bound the supremum in \eqref{eq:sup_1d} by
$$
\sup_{u_i >0}  \frac{1}{u_i} \int_{\left[x_i-\frac{u_i}{2}, x_i+\frac{u_i}{2}\right]}  e^{-c_\mu \vert y_i \vert} \diff y_i \leq \sup_{u_i >0}  \frac{1}{u_i} \int_{\left[x_i-\frac{u_i}{2}, x_i+\frac{u_i}{2}\right]} 1 \diff y_i = 1.
$$
Hence, we are able to conclude that
$m_\mu^*(x) \leq g(x)$ for all $x \in \R^2$, where $g$ is defined by
$$
g(x) \coloneqq C_\mu  \prod_{i=1,2} \min \left( 1, \frac{2}{c_\mu \vert x_i \vert} \right),
$$
and we see that $g^\beta$ is integrable with respect to $x$ for any $\beta>1$.  
\end{proof}
 
\begin{rem}
We briefly point out why the continuity condition of the density function~$f_\mu$ is essential in the point scaling regime, in particular in Lemma~\ref{lem:m_mu}~(i).
If the boxes $B\left( x, \begin{mysmallmatrix} n_1(\rho) u_1 \\ n_2(\rho) u_2 \end{mysmallmatrix}\right)$ had been \emph{nicely shrinking sets} in the sense of \cite[p.\ 140]{rudin1987complex}, the condition $f_\mu \in L^1(\R^2)$ would have been sufficient instead of requiring continuity (see Theorem~7.10 in \cite{rudin1987complex}).
In short, the crucial point for shrinking sets in order to be a sequence of nicely shrinking sets is that each set must occupy at least a certain portion of some spherical neighbourhood.
For example, a shrinking grain in the random balls model, where the size of a grain (with predetermined shape) depends only on a \emph{single} distribution, is nicely shrinking.
In contrast, the boxes $B\left( x, \begin{mysmallmatrix} n_1(\rho) u_1 \\ n_2(\rho) u_2 \end{mysmallmatrix}\right)$ in the proof of Theorem~\ref{thm:lowpoint}, where we apply Lemma~\ref{lem:m_mu}~(i), are \emph{not} nicely shrinking sets because the length-to-width ratio of the boxes tends to infinity there.
Hence, we assume in Definition~\ref{def:M_lowpoint} that the density function~$f_\mu$ is continuous such that Lemma~\ref{lem:m_mu}~(i) holds.
\end{rem}

\section{Proofs of the main results}\label{sec:proofmain}
Due to the linearity of the mapping $\mu \mapsto \widetilde{J}_\rho(\mu)$ as well as the linearity of the limiting random fields $Z$, $J_I$, $Y$, $J_L$, $S_{\gamma_1}$ and $X$,
the convergence of the finite-dimensional distributions of the centred and renormalised versions of ${J}_\rho$ is equivalent to the convergence of the one-dimensional distributions.
This can be seen using the Cram\'er-Wold device.
Therefore, we only have to deal with the convergence of the characteristic function (w.l.o.g.\ at 1) $\E \exp\left(i\frac{\widetilde{J}_\rho(\mu)}{n_\rho}\right).$
The strategy of the following proofs is similar to \cite{Bierme2010} and \cite{kaj2007}.
As mentioned above, we use $c$ and $C$ for constants which can differ from line to line and we often make use of the function $\Psi$ defined in~\eqref{def:psi}.

\subsection{Intermediate intensity regime}
\begin{proof}[Proof of Theorem \ref{thm:inter}]
We recall the characteristic function of $\widetilde{J}_\rho(\mu)$ 
$$
\E \left(e^{i\widetilde{J}_\rho(\mu)} \right) = \exp \left( \int_{\R_+^2}\int_{\R^2} \Psi(\mu(B(x,u))) \lambda_\rho \diff x F_\rho(\diff u) \right).
$$
The characteristic function of $J_I(\mu)$ is given by
\begin{equation}\label{eq:cf_intermediate}
 \E \left(e^{iJ_I(\mu)} \right) = \exp \left( \int_{\R^2 \times \R_+^2} \Psi(\mu(B(x,u))) \frac{1}{u_1^{\gamma_1+1}}  \frac{1}{u_2^{\gamma_2+1}} \diff (x,u) \right).
\end{equation}
First, we define the function $\widetilde{\phi}$ by
$$
\widetilde{\phi}(u) \coloneqq \int_{\R^2} \Psi(\mu(B(x,u))) \diff x, \quad \text{for $u \in \R_+^2$.}
$$
We note that one can verify similar to Lemma~6 in~\cite{kaj2007} that $\widetilde{\phi}$ is continuous.
Using $\vert \Psi(v) \vert \leq \frac{v^2}{2}$ and \eqref{eq:phibound}, there are constants $C>0$ and $\alpha_i$ with $\gamma_i < \alpha_i \leq 2$ for $i=1,2$ such that
$$
\vert \widetilde{\phi}(u) \vert \leq C \min\left(u_1,u_1^{\alpha_1} \right) \min\left(u_2,u_2^{\alpha_2} \right).
$$
Now, we apply Lemma \ref{lem:intasympf} with $g \coloneqq \widetilde{\phi}$ to obtain
\begin{equation}\label{eq:inter_generala}
 \int_{\R_+^2} \widetilde{\phi}(u) F_\rho(\diff u) \sim \rho^{\gamma_1+\gamma_2} \int_{\R_+^2} \widetilde{\phi}(u) \frac{1}{u_1^{\gamma_1+1}}  \frac{1}{u_2^{\gamma_2+1}} \diff u.
\end{equation}
Using this and the scaling $\lambda_\rho \rho^{\gamma_1+\gamma_2} \to 1$ shows the assertion.
\end{proof}

\begin{rem}\label{rem:inter_a}
In the general case, let us say $\lambda_\rho \rho^{\gamma_1+\gamma_2} \to a^{2-\gamma_1-\gamma_2} \in (0,\infty)$ with $a>0$ as $\rho \to 0$, the limiting compensated Poisson integral equals $J_I(\mu_a)$, where $\mu_a(\cdot) \coloneqq \mu\left(a^{-1} \, \cdot\right)$.
To see this, one can apply Theorem~\ref{thm:inter} to $\widetilde{J}'_\rho(\cdot)$ where $\lambda'_\rho \coloneqq \lambda_\rho / a^{2-\gamma_1-\gamma_2}$.
Then, the result follows after an appropriate substitution.
\end{rem}

\subsection{High intensity regime}
\begin{proof}[Proof of Theorem \ref{thm:high}]
For the sake of simplicity, we introduce 
$$
\phi_\rho (u) \coloneqq \int_{\R^2}  \Psi\left(\frac{\mu(B(x,u))}{n_\rho} \right) \diff x, \quad \text{for $u \in \R_+^2$,}
$$
with 
$n_\rho \coloneqq \sqrt{\lambda_\rho \rho^{\gamma_1+\gamma_2}}$
and recall that the characteristic function of
$
\frac{\widetilde{J}_\rho(\mu)}{n_\rho}
$
is given by
$$
\exp \left( \int_{\R_+^2} \phi_\rho (u)  \lambda_\rho F_\rho(\diff u) \right).
$$
The goal is to show the convergence of this characteristic function to
$$
\exp \left( - \frac{1}{2} \int_{\R^2 \times \R_+^2} \mu(B(x,u))^2 \frac{1}{u_1^{\gamma_1+1}} \frac{1}{u_2^{\gamma_2+1}} \diff (x,u) \right),
$$
which corresponds to a centred Gaussian random variable.
The covariance function given in \eqref{eq:cov_high_gaussian} can then be obtained by the linearity of $Z$.

Since by assumption $n_\rho \to \infty$ as $\rho \to 0$, we know that $\Psi\left(\frac{\mu(B(x,u))}{n_\rho} \right) $ can be approximated by $-\frac{1}{2} \left(\frac{\mu(B(x,u))}{n_\rho} \right)^2.$
To be more precise, we write
\begin{equation}\label{eq:intsplit}
 \int_{\R_+^2} \phi_\rho (u)  \lambda_\rho F_\rho(\diff u) = - \frac{1}{2} \int_{\R_+^2} \phi (u)\frac{\lambda_\rho}{n_\rho^2}   F_\rho(\diff u) + \int_{\R_+^2} \Delta_\rho(u) F_\rho(\diff u),
\end{equation}
where $\phi$ is given in \eqref{eq:phi} and
\begin{equation*}
 \begin{split}
 \Delta_\rho(u)  \coloneqq & \, \phi_\rho(u) \lambda_\rho+ \frac{1}{2}   \phi(u) \frac{\lambda_\rho}{n_\rho^2}\\
 = &\, \lambda_\rho \int_{\R^2}\left( \Psi\left(\frac{\mu(B(x,u))}{n_\rho} \right) + \frac{1}{2} \left( \frac{\mu(B(x,u))}{n_\rho} \right)^2 \right) \diff x.
 \end{split}
\end{equation*}
Using Lemma \ref{lem:intasympf} together with \eqref{eq:phibound}, the first integral on the right hand side of \eqref{eq:intsplit} converges to
$\int_{\R_+^2} \phi(u) \frac{1}{u_1^{\gamma_1+1}}  \frac{1}{u_2^{\gamma_2+1}} \diff u$.
Here, we refer again to Lemma~6 in~\cite{kaj2007} in order to check the continuity of $\phi$.

It remains to show that the second integral on the right hand side of~\eqref{eq:intsplit} converges to zero.
For this purpose, we show that $\Delta_\rho$ satisfies the assumptions on $g_\rho$ in Lemma~\ref{lem:intasymp2f}.

First, one can show that the estimates $\left\vert \Psi(v) + \frac{v^2}{2}\right\vert \leq \vert v \vert ^3$ and 
$$
\int_{\R^2} \vert \mu(B(x,u))\vert^3 \diff x \leq \Vert \mu \Vert ^2 \int_{\R^2} \vert \mu(B(x,u))\vert \diff x \leq \Vert \mu \Vert ^3 u_1 u_2
$$
hold.
Therefore, we obtain
$$
\left\vert \rho^{\gamma_1+\gamma_2} \Delta_\rho(u) \right\vert = \left\vert \frac{n_\rho^2}{\lambda_\rho} \Delta_\rho(u) \right\vert \leq \frac{\Vert \mu \Vert^3}{n_\rho}  u_1 u_2 \to 0
$$
as $\rho \to 0$, which shows that the first assumption of Lemma \ref{lem:intasymp2f} is satisfied.
Using $\vert \Psi(v)\vert \leq \frac{v^2}{2}$ and \eqref{eq:phibound}, the second assumption is also satisfied because we get
\begin{equation*}
 \begin{split}
 \rho^{\gamma_1+\gamma_2}\left\vert  \Delta_\rho(u) \right\vert =  & \left\vert\frac{n_\rho^2}{\lambda_\rho} \Delta_\rho(u) \right\vert \\
 \leq & \, n_\rho^2 \int_{\R^2}\left( \left\vert \Psi\left(\frac{\mu(B(x,u))}{n_\rho} \right) \right\vert+ \frac{1}{2} \left( \frac{\mu(B(x,u))}{n_\rho} \right)^2 \right) \diff x\\
 \leq & \,  n_\rho^2 \int_{\R^2}  \left( \frac{\mu(B(x,u))}{n_\rho} \right)^2 \diff x\\
 =& \int_{\R^2} \mu(B(x,u))^2 \diff x \leq  C \min\left(u_1,u_1^{\alpha_1} \right) \min\left(u_2,u_2^{\alpha_2} \right).
 \end{split}
\end{equation*}
\end{proof}

\subsection{Low intensity regime}
\subsubsection{Points scaling regime}
\begin{proof}[Proof of Theorem \ref{thm:lowpoint}]
In a first step, we prove that 
$$
\lim_{\rho \to 0} \E \exp\left({i\frac{\widetilde{J}_\rho(\mu)}{\lambda_\rho^{{1}/{\gamma_1}} \rho^2}}\right) = \exp \left( c_2^{\gamma_1} \int_{\R^2} \int_{\R_+} \Psi(u_1 f_\mu(x)) \frac{1}{u_1^{\gamma_1+1}} \diff u_1 \diff x \right),
$$
where $c_2$ is defined in \eqref{eq:c2_gamma_stable} below.
In a second step, we show that the right hand side is the characteristic function of an integral with respect to a stable random measure.

\emph{Step 1:}
We recall that the characteristic function of $\frac{\widetilde{J}_\rho(\mu)}{\lambda_\rho^{{1}/{\gamma_1}} \rho^2}$ can be written as
\begin{equation}\label{eq:char_func_J/n}
 \exp \left( \int_{\R_+^2}\int_{\R^2}  \Psi \left( \frac{1}{\lambda_\rho^{{1}/{\gamma_1}} \rho^2} \int_{B(x,u)} f_\mu(y) \diff y\right) \lambda_\rho \diff x F_\rho(\diff u) \right). 
\end{equation}
We use the definition of $m_\mu(x,u)$ in~\eqref{eq:m_mu} and the density of the scaled measure~$F$ from \eqref{eq:Fdu} to obtain
\begin{align}\label{eq:low_int_du}
 & \int_{\R_+^2}\int_{\R^2} \Psi \left( \frac{1}{\lambda_\rho^{{1}/{\gamma_1}} \rho^2}  \int_{B(x,u)} f_\mu(y) \diff y\right) \lambda_\rho \diff x F_\rho(\diff u) \nonumber \\ 
 = & \int_{\R_+^2}\int_{\R^2} \Psi \left( \frac{u_1 u_2}{\lambda_\rho^{{1}/{\gamma_1}} \rho^2} m_\mu(x,u) \right) \lambda_\rho 
 f_1 \left( \frac{u_1}{\rho}\right) \frac{1}{\rho} f_2 \left( \frac{u_2}{\rho}\right) \frac{1}{\rho} \diff x \diff u \nonumber \\ 
 = &  \int_{\R^2 \times \R_+^2} \Psi \left( u_1 m_\mu \left (x,\begin{mysmallmatrix} \lambda_\rho^{{1}/{\gamma_1}}\rho \frac{u_1} {u_2} \\ \rho u_2 \end{mysmallmatrix} \right) \right) 
 \frac{ \lambda_\rho^{1+{1}/{\gamma_1}}}{ u_2} f_1 \left( \lambda_\rho^{{1}/{\gamma_1}} \frac{u_1}{ u_2}\right)
 f_2(u_2) \diff (x,u),
\end{align} 
where we substituted first $u_2 = \rho \widetilde{u}_2$ and then $u_1 = {\lambda_\rho^{{1}/{\gamma_1}} \rho} \frac{\widetilde{u}_1}{\widetilde{u}_2}$ in the last line.
We note that
$$
\lim_{\rho \to 0} m_\mu \left (x,\begin{mysmallmatrix} \lambda_\rho^{{1}/{\gamma_1}}\rho \frac{u_1} {u_2} \\ \rho u_2 \end{mysmallmatrix} \right) = f_\mu(x)
$$
because of Lemma \ref{lem:m_mu}~(i) and that 
\begin{equation*}
 \begin{split}
 \frac{ \lambda_\rho^{1+{1}/{\gamma_1}}}{ u_2} f_1 \left( \lambda_\rho^{{1}/{\gamma_1}} \frac{u_1}{ u_2}\right)
 = &  \, \frac{ \lambda_\rho^{1+{1}/{\gamma_1}}}{ u_2} f_1 \left( \lambda_\rho^{{1}/{\gamma_1}} \frac{u_1}{ u_2}\right) \left( \lambda_\rho^{{1}/{\gamma_1}} \frac{u_1}{ u_2}\right)^{\gamma_1+1}\left( \lambda_\rho^{{1}/{\gamma_1}} \frac{u_1}{ u_2}\right)^{-\gamma_1-1}\\
 = &  \, f_1 \left( \lambda_\rho^{{1}/{\gamma_1}} \frac{u_1}{ u_2}\right) \left( \lambda_\rho^{{1}/{\gamma_1}} \frac{u_1}{ u_2}\right)^{\gamma_1+1}
 \frac{ u_2^{\gamma_1}}{ u_1^{\gamma_1+1}} \to \frac{ u_2^{\gamma_1}}{ u_1^{\gamma_1+1}}
 \end{split}
\end{equation*}
as $\rho \to 0$ because of $\lambda_\rho^{1+{1}/{\gamma_1}} \to \infty$ and the asymptotic behaviour of $f_1$.
Therefore, the integrand in \eqref{eq:low_int_du} converges to
$$
\Psi \left( u_1 f_\mu(x) \right) \frac{1}{u_1^{\gamma_1+1}}u_2^{\gamma_1}f_2(u_2).
$$
If we can also find an integrable majorant of the integrand in \eqref{eq:low_int_du}, we obtain that
\begin{equation*}
 \begin{split}
 & \lim_{\rho \to 0} \int_{\R_+^2}\int_{\R^2} \Psi \left( \frac{1}{\lambda_\rho^{{1}/{\gamma_1}} \rho^2}  \int_{B(x,u)} f_\mu(y) \diff y\right) \lambda_\rho \diff x F_\rho(\diff u) \\
 = & \,c_2^{\gamma_1} \int_{\R^2} \int_{\R_+} \Psi(u_1 f_\mu(x)) \frac{1}{u_1^{\gamma_1+1}} \diff u_1 \diff x \\
 \end{split}
\end{equation*}
by the dominated convergence theorem, where $c_2$ is defined by
\begin{equation}\label{eq:c2_gamma_stable}
 c_2 \coloneqq \left( \int_{\R_+} u_2^{\gamma_1}f_2(u_2) \diff u_2 \right)^{{1}/{\gamma_1}}.
\end{equation}
In order to find such a majorant, one can show
\begin{equation}\label{eq:low_estimate1_phi}
 \vert \Psi(v) \vert \leq 2\min \left( \vert v \vert ,  v^2 \right)
\end{equation}
and we note that there is an $\eps >0$ with $1<\gamma_1-\eps<\gamma_1+\eps <2$ such that \eqref{eq:low_estimate^12} and \eqref{eq:low_estimate2} hold.
For all $\rho<\rho_0$ with $\rho_0$ small enough, the integrand (see~\eqref{eq:low_int_du}) is therefore dominated by
\begin{equation}\label{eq:majorantintegrandlowp}
 2\min \left(\vert u_1 \vert^{\gamma_1- \eps} , \vert u_1 \vert^{\gamma_1 + \eps} \right) ( \vert m_\mu^*(x) \vert^{\gamma_1- \eps} + \vert m_\mu^*(x) \vert^{\gamma_1+ \eps})  \frac{c_{f_1}}{u_1^{\gamma_1+1}} u_2^{\gamma_1} f_2(u_2),
\end{equation} 
where we also used the technical assumption in~\eqref{eq:addrequirelow}.
Finally, we can see that \eqref{eq:majorantintegrandlowp} is integrable because of Lemma \ref{lem:m_mu}~(ii) and $1<\gamma_1 - \eps$.

\emph{Step 2:}
We deal with the integral
\begin{equation}\label{eq:exp_stable_int}
 \int_{\R^2} \int_{\R_+} \Psi(u_1 f_\mu(x)) \frac{1}{u_1^{\gamma_1+1}} \diff u_1 \diff x. 
\end{equation}
We split the integration over $\R^2$ into $\{x \colon f_\mu(x) \geq 0 \}$ and $\{x \colon f_\mu(x) < 0 \}$ and note that $\Psi(0)=0$.
We recall ${f_\mu}_+ \coloneqq \max\left({f_\mu},0\right)$ and ${f_\mu}_- \coloneqq - \min\left({f_\mu},0\right)$.
The substitution $\widetilde{u}_1=u_1 f_\mu(x)$ shows that \eqref{eq:exp_stable_int} equals
$$
d_{\gamma_1} \Vert {f_\mu}_+ \Vert_{\gamma_1}^{\gamma_1} + \bar{d}_{\gamma_1}\Vert {f_\mu}_- \Vert_{\gamma_1}^{\gamma_1},
$$
where $\bar{d}_{\gamma_1}$ is the complex conjugate of $d_{\gamma_1} \coloneqq \int_{\R_+} \Psi(u_1) \frac{1}{u_1^{\gamma_1+1}} \diff u_1 $.
We obtain
$$
d_{\gamma_1}= \frac{\Gamma(2-\gamma_1)}{\gamma_1(\gamma_1-1)} \cos \left( \frac{\pi \gamma_1}{2}\right) \left(1-i \tan \left( \frac{\pi \gamma_1}{2} \right) \right)
$$
due to \cite[p.\ 170]{samorodnitsky2000stable}.
Therefore, we can finally conclude that
\begin{equation*}
 \begin{split}
 & \lim_{\rho \to 0} \log \E \exp \left({i\frac{\widetilde{J}_\rho(\mu)}{ c_{\gamma_1,\gamma_2} \lambda_\rho^{{1}/{\gamma_1}} \rho^2}}\right) \\
 =  & \, c_2^{\gamma_1} \left( d_{\gamma_1} \left\Vert \frac{{f_\mu}_+}{c_{\gamma_1} c_2} \right\Vert_{\gamma_1}^{\gamma_1} + \bar{d}_{\gamma_1} \left\Vert \frac{{f_\mu}_-}{c_{\gamma_1} c_2} \right\Vert_{\gamma_1}^{\gamma_1} \right) \\
 =  & - \left( \Vert {f_\mu}_+ \Vert_{\gamma_1}^{\gamma_1} + \Vert {f_\mu}_- \Vert_{\gamma_1}^{\gamma_1} \right) + i \tan \left( \frac{\pi \gamma_1}{2} \right) \left( \Vert {f_\mu}_+ \Vert_{\gamma_1}^{\gamma_1} - \Vert {f_\mu}_- \Vert_{\gamma_1}^{\gamma_1} \right) \\
 =  & -\sigma_\mu^{\gamma_1} \left(1-i \beta_\mu \tan \left( \frac{\pi \gamma_1}{2} \right) \right),
 \end{split}
\end{equation*}
where 
\begin{equation}\label{eq:c1c2_gamma_stable}
 c_{\gamma_1,\gamma_2} \coloneqq c_{\gamma_1} c_2, \quad
 c_{\gamma_1} \coloneqq \left( - \frac{\Gamma(2-\gamma_1)}{\gamma_1(\gamma_1-1)} \cos \left( \frac{\pi \gamma_1}{2}\right) \right)^{{1}/{\gamma_1}},
\end{equation}
$c_2$ is given in \eqref{eq:c2_gamma_stable} and $\sigma_\mu$, $ \beta_\mu$ are given in~\eqref{eq:sigma_mu}.
\end{proof}

\subsubsection{Poisson-lines scaling regime}
\begin{proof}[Proof of Theorem \ref{thm:lowpois}]
We recall the characteristic function of $\frac{\widetilde{J}_\rho(\mu)}{\rho}$ given in~\eqref{eq:char_func_J/n}.
We proceed as in the proof of Theorem \ref{thm:lowpoint}.
Using the definition of $m_\mu(x,u)$ in~\eqref{eq:m_mu} and the density of the scaled measure~$F$ from~\eqref{eq:Fdu}, we obtain
\begin{align}\label{eq:lowpois_int_du}
 & \int_{\R_+^2}\int_{\R^2} \Psi \left( \frac{1}{\rho}  \int_{B(x,u)} f_\mu(y) \diff y\right) \lambda_\rho \diff x F_\rho(\diff u) \nonumber \\
 = &  \int_{\R_+^2}\int_{\R^2} \Psi \left( \frac{u_1 u_2}{\rho} m_\mu(x,u) \right) \lambda_\rho 
 f_1 \left( \frac{u_1}{\rho}\right) \frac{1}{\rho} f_2 \left( \frac{u_2}{\rho}\right) \frac{1}{\rho} \diff x \diff u \nonumber \\
 = &  \int_{\R^2 \times \R_+^2} \Psi \left( u_1 u_2 m_\mu \left(x,\begin{mysmallmatrix} u_1 \\ \rho u_2 \end{mysmallmatrix} \right) \right)
 \frac{\lambda_\rho}{\rho} f_1 \left( \frac{u_1}{\rho}\right) f_2 \left( u_2 \right) \diff (x,u),
\end{align} 
where we substituted $u_2 = \rho \widetilde{u}_2$ in the last line.
We note that due to \eqref{eq:low_pointwiseconv} in Definition~\ref{def:M_lowgaussian} of the space $\mathcal{M}_{L}$ 
$$
\lim_{\rho \to 0} u_1 u_2 m_\mu \left (x,\begin{mysmallmatrix} u_1\\ \rho u_2 \end{mysmallmatrix} \right) = u_2 \int_{\left[x_1-\frac{u_1}{2},x_1+\frac{u_1}{2}\right]} f_\mu(y_1,x_2) \diff y_1
$$
(pointwise for all $(x,u) \in \R^2 \times \R_+^2$) and that
\begin{equation*}
 \begin{split}
 \frac{\lambda_\rho}{\rho} f_1 \left( \frac{u_1}{\rho}\right)
 = &  \,   \frac{\lambda_\rho}{\rho} f_1 \left( \frac{u_1}{\rho}\right) \left( \frac{u_1}{\rho}\right) ^{\gamma_1+1}  \left( \frac{\rho}{u_1}\right)^{\gamma_1+1} \\
 = &  \,  f_1 \left( \frac{u_1}{\rho}\right) \left( \frac{u_1}{\rho}\right) ^{\gamma_1+1}  \lambda_\rho  \rho^{\gamma_1}  \frac{1}{u_1^{\gamma_1+1}} \to \frac{1}{u_1^{\gamma_1+1}} \\
 \end{split}
\end{equation*}
as $\rho \to 0$ because of $1 / \rho \to \infty$, the asymptotic behaviour of $f_1$ and the fact that $\lambda_\rho  \rho^{\gamma_1} \to 1$.
Therefore, the integrand in \eqref{eq:lowpois_int_du} converges to
\begin{equation}\label{eq:low_pois_generala}
 \Psi \left(  u_2 \int_{\left[x_1-\frac{u_1}{2},x_1+\frac{u_1}{2}\right]} f_\mu(y_1,x_2)\diff y_1 \right)\frac{1}{u_1^{\gamma_1+1}} f_2 \left( u_2 \right).
\end{equation}
If we can also find an integrable majorant of the integrand in \eqref{eq:lowpois_int_du}, we obtain that
\begin{align}\label{eq:low_pois_generala_lim}
 & \lim_{\rho \to 0} \int_{\R_+^2}\int_{\R^2} \Psi \left( \frac{1}{\rho}  \int_{B(x,u)} f_\mu(y) \diff y\right) \lambda_\rho \diff x F_\rho(\diff u) \nonumber \\
 = &  \int_{\R^2 \times \R_+^2}
 \Psi \left(  u_2 \int_{\left[x_1-\frac{u_1}{2},x_1+\frac{u_1}{2}\right]} f_\mu(y_1,x_2)\diff y_1 \right)\frac{1}{u_1^{\gamma_1+1}} f_2 \left( u_2 \right) \diff (x,u)
\end{align}
by the dominated convergence theorem.
Using the estimates in~\eqref{eq:low_estimate1_phi} and~\eqref{eq:low_estimate^12}, an extended version of \eqref{eq:low_estimate2} and the technical assumption in~\eqref{eq:addrequirelow}, we see that the integrand in \eqref{eq:lowpois_int_du} is dominated by
\begin{equation}\label{eq:majorantintegrandlowpois}
 \begin{split}
 & 2 \min \left(\vert u_1 \vert^{\gamma_1- \eps} , \vert u_1 \vert^{\gamma_1 + \eps} \right)
 ( \vert u_2 \vert^{\gamma_1- \eps} + \vert u_2  \vert^{\gamma_1+ \eps}) \\
 & \times ( \vert m_\mu^*(x) \vert^{\gamma_1- \eps} + \vert  m_\mu^*(x) \vert^{\gamma_1+ \eps}) \frac{c_{f_1}}{u_1^{\gamma_1+1}} f_2(u_2)
 \end{split}
\end{equation} 
for all $\rho<\rho_0$ with $\rho_0$ small enough.
Here, we have to choose $\eps >0$ such that ${1<\gamma_1 - \eps}$, ${\gamma_1 + \eps < 2}$ as well as $\gamma_1 + \eps < \gamma_2$.
These conditions together with Lemma~\ref{lem:m_mu}~(ii) ensure that \eqref{eq:majorantintegrandlowpois} is integrable.

Since the characteristic function $\E \left(e^{iJ_L(\mu)} \right)$ of the limit $J_L(\mu)$ is given by the exponential of \eqref{eq:low_pois_generala_lim}, the convergence of the characteristic function is proven.
\end{proof}

\begin{rem}\label{rem:lowpois_a}
In the general case, let us say $\lambda_\rho \rho^{\gamma_1} \to a^{2-\gamma_1} \in (0,\infty)$ with $a>0$ as $\rho \to 0$, we obtain
$
\frac{\widetilde{J}_\rho(\mu) }{a \rho} \to J_L(\mu_{a}),
$
where we recall $\mu_a(\cdot) \coloneqq \mu\left(a^{-1} \, \cdot\right)$.
In order to prove this, we note that one gets \eqref{eq:low_pois_generala_lim} with the additional factor $a^{2-\gamma_1}$ for the logarithm of the characteristic function of the limit in the general case.
Then, one can deduce the result after an appropriate substitution.
\end{rem}

\subsubsection{Gaussian-lines scaling regime}
\begin{proof}[Proof of Theorem \ref{thm:lowgaus}]
We recall the characteristic function of $\frac{\widetilde{J}_\rho(\mu) }{{\rho}^{1-\eta / 2}}$,
which, after the substitution $u_2 = \rho \widetilde{u}_2$, equals
$$
\exp \left( \int_{\R^2 \times \R_+^2} \Psi\left( \frac{\mu \left(B \left(x,\begin{mysmallmatrix} u_1 \\ \rho u_2 \end{mysmallmatrix} \right)\right)}{{\rho}^{1-\eta / 2}}  \right)
\frac{ \lambda_\rho}{\rho} f_1 \left( \frac{u_1}{\rho} \right) f_2(u_2) \diff (x,u) \right).
$$
The goal is to show for some $\sigma^2>0$ the convergence of this characteristic function to $\exp \left( - \sigma^2 /2 \right)$, which corresponds to a centred Gaussian random variable.

To be more precise, we write
\begin{align}
 & \int_{\R^2 \times \R_+^2} \Psi\left( \frac{ \mu \left(B \left(x,\begin{mysmallmatrix} u_1 \\ \rho u_2 \end{mysmallmatrix} \right)\right)}{{\rho}^{1-\eta / 2}} \right)
 \frac{ \lambda_\rho}{\rho} f_1 \left( \frac{u_1}{\rho} \right) f_2(u_2) \diff (x,u) \nonumber \\
 = & - \frac{1}{2} \int_{\R^2 \times \R_+^2} u_2^2 \left( \frac{\mu \left(B \left(x,\begin{mysmallmatrix} u_1 \\ \rho u_2 \end{mysmallmatrix} \right)\right)}{\rho u_2} \right)^2  \rho^\eta \frac{\lambda_\rho}{\rho} f_1 \left( \frac{u_1}{\rho} \right) f_2(u_2) \diff (x,u) \label{eq:intlowsplit_1} \\
 & +\int_{\R^2 \times \R_+^2} \Delta_\rho(u,x)\frac{ \lambda_\rho}{\rho} f_1 \left( \frac{u_1}{\rho} \right) f_2(u_2) \diff (x,u), \label{eq:intlowsplit_2}
\end{align}
where
\begin{equation}\label{eq:deltagaus}
 \begin{split}
 & \Delta_\rho(u,x)
 \coloneqq  \Psi\left( \rho^{\eta/2-1} \mu \left(B \left(x,\begin{mysmallmatrix} u_1 \\ \rho u_2 \end{mysmallmatrix} \right)\right) \right)+  \frac{1}{2}\left( \rho^{\eta/2-1} \mu \left(B \left(x,\begin{mysmallmatrix} u_1 \\ \rho u_2 \end{mysmallmatrix} \right)\right) \right)^2. \\
 \end{split}
\end{equation}
First, we discuss the integral in \eqref{eq:intlowsplit_2} in the case of $\gamma_2>3$.
Since we have $\left\vert \Psi(v) + \frac{v^2}{2}\right\vert \leq \vert v \vert ^3$, we can bound \eqref{eq:deltagaus} and can thus bound the integrand by
\begin{align}\label{eq:bound2integral}
 & \,  \rho^{3 \eta/2} u_2^3 \left( \frac{ \vert \mu \left(B \left(x,\begin{mysmallmatrix} u_1 \\ \rho u_2 \end{mysmallmatrix} \right)\right) \vert}{\rho u_2} \right)^3 \frac{ \lambda_\rho}{\rho} f_1 \left( \frac{u_1}{\rho} \right) f_2(u_2) \nonumber \\
 \leq & \, C_\mu^3 \rho^{\eta/2} u_2^3 
 \left( \frac{1}{\rho u_2} \int_{B \left(x,\begin{mysmallmatrix} u_1 \\ \rho u_2 \end{mysmallmatrix} \right)}  e^{-c_\mu \vert y_1 \vert} e^{-c_\mu \vert y_2 \vert} \diff y 
 \right) ^3
 \lambda_\rho \rho^{\eta-1} c_{f_1} \left( \frac{\rho}{u_1} \right)^{\gamma_1+1} f_2(u_2) \nonumber \\
 \leq & \,C \rho^{\eta/2} u_2^3 
 g_1(x_1,u_1)  ^3
 g_2(x_2) ^3
 \frac{1}{u_1^{\gamma_1+1}} f_2(u_2),
\end{align}
for $\rho<\rho_0$ with $\rho_0$ small enough, where $g_1$ is given in~\eqref{eq:g_1} and
$$
g_2(x_2) \coloneqq  \min \left( 1, \frac{2}{c_\mu \vert x_2 \vert} \right).
$$
Here, we used the assumption~\eqref{eq:decay_mu} from Definition~\ref{def:M_lowgaussian}, the technical assumption in~\eqref{eq:addrequirelow} and the fact that \mbox{$\lambda_\rho \rho^{\gamma_1+\eta} \to 1$}.
Furthermore, we used
$$
\sup_{\rho > 0}  \frac{1}{\rho  u_2} \int_{\left[x_2-\frac{\rho  u_2}{2}, x_2+\frac{\rho  u_2}{2}\right]}  e^{-c_\mu\vert y_2 \vert} \diff y_2 \leq g_2(x_2)
$$
from the proof of Lemma~\ref{lem:m_mu}~(ii).
By \eqref{eq:bound2integral}, we see that the integrand in~\eqref{eq:intlowsplit_2} has an integrable majorant since we assumed $\gamma_2 > 3$ and because $g_2^3$ is integrable with respect to $x_2$ and $g_1(x_1,u_1)^3 /{u_1^{\gamma_1+1}} $ is also integrable (in order to check this, one just has to follow the lines below~\eqref{eq:exist_lowpoiss}).
Moreover, the majorant converges to zero because of \mbox{$\rho^{\eta/2} \to 0$}.

In the case of $2 < \gamma_2 \leq 3$, we note that there is an $\eps>0$ such that $2 < \gamma_2 - \eps < 3$ as well as $\left\vert \Psi(v) + \frac{v^2}{2}\right\vert \leq \vert v \vert ^{\gamma_2 - \eps}$.
The last-mentioned estimate can be deduced from Lemma~1 in~\cite{kaj2007} by a case distinction (cf.~\eqref{eq:low_estimate^12}).
Similar to above, we can bound the integrand in \eqref{eq:intlowsplit_2} by
\begin{align}\label{eq:bound2integral_extended}
 & \, \rho^{{(\gamma_2 - \eps)} \eta/2} \left( \frac{ \vert \mu \left(B \left(x,\begin{mysmallmatrix} u_1 \\ \rho u_2 \end{mysmallmatrix} \right)\right) \vert}{\rho} \right)^{\gamma_2 - \eps} \frac{ \lambda_\rho}{\rho} f_1 \left( \frac{u_1}{\rho} \right) f_2(u_2) \nonumber \\
 \leq & \, \rho^{{(\gamma_2 - \eps-2+2)} \eta/2} u_2^{\gamma_2-\eps} \left( \frac{\vert \mu \left(B \left(x,\begin{mysmallmatrix} u_1 \\ \rho u_2 \end{mysmallmatrix} \right)\right) \vert}{\rho u_2} \right)^{\gamma_2-\eps} \frac{ \lambda_\rho}{\rho} c_{f_1} \left( \frac{\rho}{u_1} \right)^{\gamma_1+1} f_2(u_2) \nonumber \\
 \leq & \, C  \rho^{{(\gamma_2 - \eps-2)} \eta/2} u_2^{\gamma_2-\eps} 
 \left(
 g_1(x_1,u_1) 
 g_2(x_2)
 \right)^{\gamma_2-\eps}
 \lambda_\rho \rho^{\gamma_1+\eta}  \frac{1}{u_1^{\gamma_1+1}} f_2(u_2) \nonumber \\
 \leq & \,C \rho^{{(\gamma_2 - \eps-2)} \eta/2} u_2^{\gamma_2-\eps}
 g_1(x_1,u_1) ^{\gamma_2-\eps}
 g_2(x_2)^{\gamma_2-\eps}
 \frac{1}{u_1^{\gamma_1+1}} f_2(u_2)
\end{align}
for $\rho<\rho_0$ with $\rho_0$ small enough.
Using $\gamma_1<\gamma_2-\eps$, we can see by \eqref{eq:bound2integral_extended} that the integrand in \eqref{eq:intlowsplit_2} has an integrable majorant because $g_2^{\gamma_2-\eps}$ and $g_1(x_1,u_1)^{\gamma_2-\eps}/{u_1^{\gamma_1+1}}$ are integrable (with the same reasons as above) and that it converges to zero because of ${\gamma_2-\eps-2}>0$.

Therefore, we obtain in both cases that the integral in \eqref{eq:intlowsplit_2} converges to zero by the dominated convergence theorem.

Next, we deal with the integral in \eqref{eq:intlowsplit_1} and show that it converges to
\begin{equation}\label{eq:sigma2lowgaus}
 \sigma^2 \coloneqq \int_{\R^2 \times \R_+^2} u_2^2 \left( \int_{\left[x_1-\frac{u_1}{2},x_1+\frac{u_1}{2}\right]} f_\mu(y_1,x_2) \diff y_1 \right)^2 \frac{f_2(u_2)}{u_1^{\gamma_1+1}}   \diff(x,u).
\end{equation}
The convergence of the integrand can be seen similar to above using Definition~\ref{def:M_lowgaussian} of the space $\mathcal{M}_{L}$, the asymptotic behaviour of $f_1$ and the fact that \mbox{$\lambda_\rho \rho^{\gamma_1+\eta} \to 1$}.
A majorant of the integrand is given by
\begin{equation*}
 C u_2^2  g_1(x_1,u_1) ^2 g_2(x_2)^2  \frac{1}{u_1^{\gamma_1+1}}  f_2(u_2), 
\end{equation*}
which is integrable for $\gamma_2>2$.
Applying the dominated convergence theorem, the convergence of the characteristic function is proven.
By linearity, the covariance function given in \eqref{eq:cov_low_gaussian} follows from \eqref{eq:sigma2lowgaus}.
\end{proof}

\begin{rem}\label{rem:lowgaus_a}
In the general case, let us say $\lambda_\rho \rho^{\gamma_1+\eta} \to a^{2} \in (0,\infty)$ with $a>0$ as $\rho \to 0$, the limit is a centred Gaussian linear random field which is given by $( Y(a \mu))_\mu$, where $a \mu$ has the density $a f_\mu$ and the variance of $ Y(a \mu)$ is just $ a^2 \sigma^2$.
This can be seen since we obtain the additional factor $a^{2}$ in~\eqref{eq:sigma2lowgaus}.
\end{rem}

\subsection{The finite variance case}
\begin{proof}[Proof of Theorem \ref{thm:finite_variance}]
We use the definition of $m_\mu(x,u)$ in \eqref{eq:m_mu} to obtain for the logarithm of the characteristic function of $\frac{\widetilde{J}_\rho(\mu)}{\rho^2 \sqrt{ \lambda_\rho v_1 v_2}}$
\begin{equation*}\label{eq:finite_var_int_du}
 \begin{split}
 & \, \int_{\R_+^2}\int_{\R^2} \Psi \left( \frac{1}{\rho^2 \sqrt{ \lambda_\rho v_1 v_2}}  \int_{B(x,u)} f_\mu(y) \diff y\right) \lambda_\rho \diff x F_\rho(\diff u) \\
 = & \, \int_{\R_+^2}\int_{\R^2} \Psi \left( \frac{u_1 u_2}{\rho^2 \sqrt{ \lambda_\rho v_1 v_2}} m_\mu(x,u) \right) \lambda_\rho \diff x F_\rho(\diff u) \\
 = & \, \int_{\R_+^2}\int_{\R^2} \Psi \left(\frac{u_1 u_2}{\sqrt{ \lambda_\rho v_1 v_2}} m_\mu \left (x,\begin{mysmallmatrix} \rho u_1 \\ \rho u_2 \end{mysmallmatrix} \right) \right) \lambda_\rho \diff x F(\diff u),
 \end{split}
\end{equation*} 
where we substituted $u_2 = \rho \widetilde{u}_2$ and $u_1 = \rho \widetilde{u}_1$ in the last line.
We note that
$$
\lim_{\rho \to 0} m_\mu \left (x,\begin{mysmallmatrix} \rho u_1\\ \rho u_2 \end{mysmallmatrix} \right) = f_\mu(x)
$$
because of Lemma \ref{lem:m_mu}~(i).
Due to the estimate $\left\vert \Psi(v) + \frac{v^2}{2}\right\vert \leq \vert v \vert ^3$ and $\lambda_\rho \to 0$, we get
$$
\lim_{\rho \to 0} \Psi \left(\frac{u_1 u_2}{\sqrt{ \lambda_\rho v_1 v_2}} m_\mu \left (x,\begin{mysmallmatrix} \rho u_1 \\ \rho u_2 \end{mysmallmatrix} \right) \right) \lambda_\rho
=-\frac{u_1^2 u_2^2 f_\mu(x)^2}{2 v_1 v_2}.
$$
Furthermore, we use $\left\vert \Psi(v)\right\vert \leq \frac{v^2}{2}$ and the definition of $m_\mu^*(x)$ in \eqref{eq:m_mu^*} to obtain
$$
\Psi \left(\frac{u_1 u_2}{\sqrt{ \lambda_\rho v_1 v_2}} m_\mu \left (x,\begin{mysmallmatrix} \rho u_1 \\ \rho u_2 \end{mysmallmatrix} \right) \right) \lambda_\rho
\leq \frac{u_1^2 u_2^2 m_\mu^*(x)^2}{2 v_1 v_2}.
$$
Since the right hand side can serve as an integrable majorant, we can apply the dominated convergence theorem and obtain
$$
\lim_{\rho \to 0} \E \exp \left({i\frac{\widetilde{J}_\rho(\mu)}{{\rho^2 \sqrt{ \lambda_\rho v_1 v_2}}}} \right) = \exp \left(-\frac{1}{2} \int_{\R^2} f_\mu(x)^2 \diff x\right),
$$
which is the characteristic function of a Gaussian random variable.
Finally, we can conclude that the limiting random field is a centred Gaussian linear random field with covariance function given in \eqref{eq:cov_thm_fin_var}.
\end{proof}

\paragraph{Acknowledgement.}
The work of S.\ Schwinn is supported by the `Excellence Initiative' of the German Federal and State Governments and the Graduate School of Computational Engineering at Technische Universit\"at Darmstadt.

\end{document}